\documentclass[11pt,reqno]{amsart}

\usepackage[american]{babel}
\usepackage{comment}
\usepackage{amsmath}
\usepackage{dsfont,mathtools,amssymb}
\usepackage{color}
\usepackage[hidelinks]{hyperref}
\mathtoolsset{showonlyrefs,showmanualtags}
\usepackage[toc]{appendix}

\newtheorem{theorem}{Theorem}[section]
\newtheorem{proposition}[theorem]{Proposition}
\newtheorem{lemma}[theorem]{Lemma}
\newtheorem{corollary}[theorem]{Corollary}
\newtheorem{remark}[theorem]{Remark}

\newtheorem*{acknowledgement}{Acknowledgement}
\numberwithin{equation}{section}
\numberwithin{figure}{section}

\newcommand{\E}{\mathds{E}}

\newcommand{\Prob}{\mathds{P}}

\renewcommand{\rho}{\varrho}

\newcommand{\R}{\mathbb{R}}
\newcommand{\N}{\mathbb{N}}
\newcommand{\IND}{{\bf 1}}
\newcommand{\dd}{{\rm d}}

\DeclareMathOperator{\var}{Var}

\newcommand{\eps}{\varepsilon}

	\newcommand{\cA}{\ensuremath{\mathcal A}} 
	\newcommand{\cB}{\ensuremath{\mathcal B}}

	\newcommand{\cE}{\ensuremath{\mathcal E}} 
	\newcommand{\cF}{\ensuremath{\mathcal F}} 
	\newcommand{\cG}{\ensuremath{\mathcal G}}

	\newcommand{\cJ}{\ensuremath{\mathcal J}}

	\newcommand{\cM}{\ensuremath{\mathcal M}}

	\newcommand{\cQ}{\ensuremath{\mathcal Q}} 
	 
	\newcommand{\cS}{\ensuremath{\mathcal S}}

	\newcommand{\cX}{\ensuremath{\mathcal X}} 
	\newcommand{\cY}{\ensuremath{\mathcal Y}}

	\newcommand{\bbE}{{\ensuremath{\mathbb E}} }

	\newcommand{\bbN}{{\ensuremath{\mathbb N}} }

	\def\({\left(}
	\def\){\right)}
	\newcommand{\fstop}{\; \text{.}}
	\newcommand{\comma}{\; \text{,}\;\;}
	
	\renewcommand{\complement}{c}
	
\allowdisplaybreaks
	
	
\author{Pietro Caputo, Matteo Quattropani, Federico Sau}
\address{Pietro Caputo\\ Universit\`{a} degli Studi di Roma Tre}
\email{pietro.caputo@uniroma3.it}
\address{Matteo Quattropani\\ Sapienza Universit\`{a} di Roma}
\email{matteo.quattropani@uniroma1.it}
\address{Federico Sau\\ Università degli Studi di Trieste}
\email{federico.sau@units.it}

\begin{document}
		\title[Repeated Block Averages]{Repeated Block Averages: \\ entropic time and mixing profiles}
		\maketitle
		\begin{abstract}
		We consider randomized dynamics over the $n$-simplex, where at each step a random set, or \emph{block}, of coordinates is evenly averaged. When all blocks have size 2, this reduces to the \emph{repeated averages} studied in \cite{chatterjee2020phase}, a version of the \emph{averaging process} on a graph \cite{aldous_lecture_2012}. We study the convergence to equilibrium of this process as a function of the distribution of the block size, and provide sharp conditions for the emergence of the \emph{cutoff phenomenon}.
Moreover, we characterize the size of the cutoff window and provide an explicit Gaussian cutoff profile. To complete the analysis, we study in detail the simplified case where the block size is not random. We show that the absence of a cutoff is equivalent to having blocks of size $n^{\Omega(1)}$, in which case we provide a convergence in distribution for the total variation distance at any given time, showing that, on the proper time scale, it remains constantly 1 up to an exponentially distributed random time, after which it decays following a Poissonian profile.
		\end{abstract}
	
		\setcounter{tocdepth}{1}
	\thispagestyle{empty}
			\section{Introduction}\label{sec:intro}
			The following process is commonly referred to as Repeated Averages. Given a probability vector $\eta=(\eta(1),\dots,\eta(n))$, one step of the dynamics consists in picking two distinct coordinates $x,y$ uniformly at random and replacing both $\eta(x)$ and $\eta(y)$ by $(\eta(x)+\eta(y))/2$. Iteration of this single step defines a discrete time, random dynamical system on the $n$-simplex, converging with high probability to the fixed point given by the uniform distribution $\pi(x)\equiv 1/n$.  Despite the simplicity of 		this process, a precise understanding of its convergence was obtained only very recently in \cite{chatterjee2020phase}, where the process was shown to exhibit a cutoff phenomenon at time $t = c\,n\log n$, with $c=1/2\log 2$, and with an explicit gaussian mixing profile in a  time window 
			of size $n\sqrt{\log n}$. We refer to \cite{chatterjee2020phase} for a precise description of the result and for more background.
			
			It is natural to ask what happens when the two--coordinate averaging at each step is replaced by a $k$-coordinate averaging, where $k\ge 2$ is an integer. We refer to this as the Repeated Block Average (or, simply, Block Average) process. While an adaptation of the techniques in   \cite{chatterjee2020phase} readily implies that, for fixed $k\in\bbN$, the analogous cutoff phenomenon takes place at time $t = c\,n\log n$ with $c = 1/k\log k$, cases where $k$ is itself a random variable and/or is allowed to grow with $n$ provide a richer phenomenology.

			In this paper we develop a robust approach to the Block Average process, which allows us to derive a complete description of the total variation distance from stationarity in such contexts, including sharp criteria for the occurrence of cutoff and explicit descriptions of the mixing profile within the cutoff window in terms of the convolution of two normal random variables.  Our approach is undoubtedly inspired by the methods introduced in \cite{chatterjee2020phase}, but our analysis departs from that work in several key places. In particular, we bypass the discretization technique, a crucial ingredient in the proof of the main results of \cite{chatterjee2020phase}, by viewing the Block Average process as the projection of a richer fragmentation process, which we call the {\em pile dynamics}, and which encodes the history of all averaging events.

			If the  block sizes at each step of the dynamics are i.i.d.\ copies of an integer  random variable $X\ge 2$, we show that whenever cutoff occurs, then it takes place at the {\em entropic time} 
			\begin{equation}\label{eq:t-ent0}
				t_{\rm ent} = \frac{n\log n}{\bbE[X\log X]}\fstop
			\end{equation}
			Intuitively,
			$t_{\rm ent}$ is the number of steps required by the Block Average process started at a Dirac mass to reach an entropy of order $\log n$.
			It can be seen 
			as the analogue of the entropic time  introduced in \cite{bordenave_cutoff2019} in the context of random walks on random graphs. In line with this perspective, we link the occurrence of cutoff to a Weak Law of Large Numbers for $\mu^{-1}\log n$ i.i.d.\ copies of $\log Y$, where $Y$ is the size-biased version of the block-size $X$, and $\mu=\bbE\big[\log Y\big]$ is its mean, 
			provided that $\mu^{-1}\log n\gg 1$. 
			We shall refer to this asymptotic inequality, which turns out to be necessary for cutoff to occur, as the \textit{entropic product condition} (Proposition \ref{pr:prod-cond} and Theorem \ref{th:cutoff}).
			
			We strengthen this result by determining the corresponding cutoff window and Gaussian profile in the case where  $\log Y$ additionally satisfies a Central Limit Theorem. Interestingly, we describe our profiles as arising from the interplay of two distinct sources of randomness. Roughly speaking, a source of Gaussian fluctuations arises from the randomness of the block sizes, while another stems from the random choice of the elements forming the block (Theorems \ref{th:cutoff-window} and \ref{th:profile1}).

	When the block size $X$ is deterministically equal to $k=k(n)$, we show that cutoff occurs if and only if $k=n^{o(1)}$. In other terms,  the product condition $\mu^{-1}\log n\gg 1$ 
	mentioned above becomes a necessary and sufficient condition for the occurrence of 
	cutoff in this case (Corollary \ref{cor:cutoff-deterministic}).
	On the other hand, when $k=n^{\Omega(1)}$, the total-variation distance does not necessarily concentrate. We study this case by a suitable multi-scale analysis for Block Averages which successfully captures its limiting Poisson-like profile (Theorem \ref{th:polynomial}).

			We end this introduction with some comments on related work.
			The two-coordinate average process analyzed in \cite{chatterjee2020phase} is the complete graph (or mean field) version of the so called Averaging Process on a graph $G=(V,E)$, defined as above with the difference that the two coordinates $x,y$ are now picked uniformly at random among the edges of $G$, that is, the unordered pairs $\{x,y\}\in E$. The systematic mathematical study of the Averaging Process on a graph has been initiated in \cite{aldous_lecture_2012}, where some general facts about convergence in $L^2$ and in relative entropy where obtained. Repeated averages along the edges of a graph represents a natural toy model for stochastic processes on graphs with a richer structure, with remarkable applications to opinion dynamics, consensus finding, gossip algorithms, and even quantum computing, see \cite{chatterjee2020phase,movassagh_repeated2022} for further motivation and background. We also refer to \cite{movassagh_repeated2022}, \cite{quattropani2021mixing}, \cite{caputo_quattropani_sau_cutoff_2023}, and \cite{sau_concentration_2024} for further general results and for some mixing analysis on certain specific families of graphs. 
			The Averaging process can also be studied on hypergraphs, so that a single step now consists in a full average of all the masses on the vertices in a given hyperedge.  This point of view was taken in \cite{spiro_averaging_2022}, where the $L^2$ convergence was analyzed. We note that the Block Average process introduced above can be seen as a mean field version of the  Averaging Process on a hypergraph. Convergence in relative entropy for  Block Averages can be deduced from the estimates in \cite{bristiel_caputo_entropy_2021}.

			\section{Model and main results}\label{sec:results}			
			\subsection{Block Average process}
			Fix an  integer $n\ge 2$, and write $V=[n]=\{1,\dots,n\}$. Let $\Delta$ denote the set of probability vectors $\eta=(\eta(x))_{x\in V}$.  We view $\eta$ as a mass distribution over the vertex set $V$. Given $A\subset V$, $\eta\in\Delta$, we write $\varPsi_A\eta\in\Delta$ for the mass distribution obtained after averaging in $A$,
			\begin{equation}
				x\in V\longmapsto 	
				\varPsi_A\eta(x)= \begin{cases}\textstyle{\frac1{|A|}\sum_{y\in A}\eta(y)} &\text{if}\ x\in A\\
					\eta(x) &\text{else}\fstop		
				\end{cases}
			\end{equation}
			Given a random variable $X$ with values in $\{2,\dots,n\}$, let $X_1,X_2,\dots$ denote i.i.d.\  copies of $X$ and let $A_1,A_2,\dots$ denote independent choices of uniformly random subsets such that $A_t\subset V$ has size $|A_t|=X_t$.   The $X$-Block Average process is defined as the discrete time evolution $\eta_1,\eta_2,\dots$ given by
			\begin{equation}\label{def:XBA}
				\eta_t = 	\varPsi_{A_t}\varPsi_{A_{t-1}}\cdots \varPsi_{A_1}\eta\fstop
			\end{equation}
			We consider the total variation distance from stationarity
			\begin{equation}d_{\rm TV}(t) =
				\|\eta_t-\pi\|_{\rm TV}\comma 
			\end{equation}
			where $\|v\|_{\rm TV}=\frac12\sum_{x\in V}|v(x)|$, and $\pi$ denotes the uniform distribution $\pi(x)\equiv 1/n$. It is not hard to check that for any initial mass distribution $\eta\in \Delta$, the sequence of random variables $d_{\rm TV}(t)$, $t\in\bbN$, is almost surely non-increasing and converging to $0$, and that the worst initial condition is obtained by placing all the mass in a single site $x_0\in V$; see,  e.g., \cite{caputo_quattropani_sau_cutoff_2023}. From now on, unless otherwise stated, we assume that the initial state is such a Dirac mass. 
			
			Our goal is to quantify this convergence in the limit $n\to \infty$.  
			To this end, we define the \emph{$\varepsilon$-mixing time} of the process as
			\begin{equation}
				t_{\rm mix}(\varepsilon) = 
				\inf\{t\in\bbN:\;  
				\E\left[d_{\rm TV}(t)\right]<\varepsilon \}\comma
			\end{equation}
			where $\E$ denotes expectation over the Block Average process. 
			The process is said to exhibit the {\em cutoff phenomenon} if the mixing time is, to leading order, insensitive of the value of $\varepsilon$, that is, if 
			\begin{equation}\label{eq:cutoff-def}
				t_{\rm mix}(\varepsilon)\sim t_{\rm mix}(1-\varepsilon)
				\comma\qquad \varepsilon\in(0,1)\comma
			\end{equation}
			where, for two positive sequences $a=a(n), b=b(n)$, we write $a\sim b$ for $a(n)/b(n) \to 1$.
			We shall also adopt the standard asymptotic notation $a \ll b$ for $a(n)/b(n) \to 0$, and $a \asymp b$ if $a(n)/b(n)$ and $b(n)/a(n)$ are both bounded.  Similarly, we write $a\lesssim b$ if $a(n)/b(n)$ stays bounded.
			
			\subsection{Entropic time and other timescales}\label{sec:times} We shall identify the entropic time \eqref{eq:t-ent0} as the relevant mixing timescale for  Block Averages, regardless of the block size distribution, see Theorem \ref{pr:tent-tmix}.
			Before introducing this result, we review and introduce some other relevant timescales for the process.

			As in the usual Averaging Process \cite{aldous_lecture_2012}, a simple duality relates the mass distribution $\eta_t$ to the corresponding  random walk.
			In our case, a direct computation, see Appendix \ref{app:duality}, shows that if  the process is started at the Dirac mass at $x_0\in V$, then for all $t\in\bbN$ and $x\in V$, 
			\begin{equation}\label{eq:duality}
				\E[\eta_t(x)] = P_{\rm RW}^t(x_0,x)\comma
			\end{equation}
			where $P_{\rm RW}$ is the transition matrix of the random walk on $V$ obtained as follows. If the current state is $x$, the next state $y$ is obtained by first sampling an integer $k$ distributed as the random variable $X$, then a uniformly random subset $A\subset V$ with size $k$, and finally taking $y\in A$ uniformly at random if $x\in A$, otherwise $y=x$ if $x\notin A$.  
			It is straightforward  to check that $P_{\rm RW}$ has one eigenvalue $1$, while all remaining ones are equal to $1-\frac{\bbE[X]-1}{n-1}$. In particular, the inverse of the spectral gap, 
			referred to as the {\em relaxation time} is given by (see Appendix \ref{app:duality}) 
			\begin{equation}\label{eq:t-rel}
				t_{\rm rel}= 
				\frac{n-1}{\E[X]-1}\fstop
			\end{equation} 
			Observe that $t_{\rm rel}\in [1,n-1]$, and the extremes $1$, and $n-1$ are reached only in case $X$ is deterministically equal to $n$ and $2$, respectively. The relaxation time controls convergence to  stationarity of the Block Average process in $L^2$ norm, see \cite{aldous_lecture_2012}
			for the usual averaging process on graphs, and see \cite{spiro_averaging_2022} for the hypergraph version. In our setup, this can be formulated as follows. 
			\begin{proposition}
				\label{prop:L2}
				For any initial distribution $\eta\in\Delta$, the $X$-Block Average process satisfies 
				\begin{equation}
					\E\bigg[\bigg\|\frac{\eta_t}{\pi}-1\bigg\|_2^2\bigg]  = \bigg(1-\frac1{t_{\rm rel}}\bigg)^t\, \bigg\|\frac{\eta}{\pi}-1\bigg\|_2^2
					 \comma\qquad t \in\bbN\comma
				\end{equation}
				where $\|v\|_2:= (\sum_{x\in V}\pi(x)|v(x)|^2)^{\frac12}$ for all $v\in \R^V$.
			\end{proposition}
			For the reader's convenience, we record the proof of this estimate in the Appendix \ref{app:duality}.
			Using Schwarz' inequality and the fact that $\|\frac{\eta}{\pi}-1\|_2^2\le n-1$, Proposition \ref{prop:L2} implies a preliminary upper bound on mixing times:
			\begin{equation}\label{l1l2bound}
				t_{\rm mix}(\varepsilon)\le 	t_{\rm rel}\,\big(\log n - 2\log \eps\big)\ ,\qquad \varepsilon \in (0,1)\ .
			\end{equation}

			An alternative upper bound can be obtained by using relative entropy. Let $t_{\rm ent}$ denote the entropic time defined in \eqref{eq:t-ent0}, and let $D(\eta\|\pi)$ denote the relative entropy (or KL divergence) of $\eta\in\Delta$ with respect to $\pi$. The following is a direct consequence of  \cite[Theorem 1.1]{bristiel_caputo_entropy_2021}.
			\begin{proposition}\label{pr:BC}
			For all $t\in\bbN$, all initial $\eta\in\Delta$, 
			\begin{equation}
				\E[D(\eta_t\|\pi)]\le \bigg(1-\frac1{t_{\rm ent}}\bigg)^tD(\eta\|\pi)\fstop
			\end{equation}
			\end{proposition}
			\begin{proof}
			An equivalent formulation of \cite[Theorem 1.1]{bristiel_caputo_entropy_2021} states that if  $A_1$ is a random subset as in \eqref{def:XBA}, then, for all
			$\eta\in\Delta$,
			\begin{equation}
			\E[D(\varPsi_{A_1}\eta\|\pi)]\le (1-\kappa) D(\eta\|\pi)\comma\qquad \kappa = \frac1{t_{\rm ent}}\fstop	
			\end{equation}
			Iteration of the above estimate yields the conclusion.  
			\end{proof}
			Since $D(\eta\|\pi)\le \log n$,  Pinsker inequality and Proposition \ref{pr:BC} imply the upper bound 
			\begin{equation}\label{ent_bound}
			t_{\rm mix}(\varepsilon)\le 2	t_{\rm ent}\,\big(\log \log n - \log (\sqrt 2\eps)\big)\ ,\qquad \varepsilon \in (0,1)\ .
			\end{equation}
			\begin{remark}
			The entropic time and the relaxation time always satisfy
			\begin{equation}
				1\le	\frac{t_{\rm ent}}{t_{\rm rel}}\le \, \frac{n}{n-1}\frac{\E[X]-1}{\E[X]}
				\frac{\log n}{\log \bbE[X]}\ .
			\end{equation}
			Indeed, the first inequality follows by integrating $(X-1)\, n\log n-(n-1)\, X\log X\ge 0$, while the second is
			Jensen inequality $\E[X\log X]\ge \E[X]\log \E[X]$.
			\end{remark}

			Following \cite{movassagh_repeated2022}, relative entropy may be used also to obtain a preliminary lower bound. 
			\begin{proposition}\label{pr:movas}
			For all $t\in\bbN$, all initial $\eta\in\Delta$, 
			\begin{equation}
				\E[D(\eta_t\|\pi)]\ge D(\eta\|\pi)-\frac{t\,\log n}{ t_{\rm ent}}\fstop
			\end{equation}
			\end{proposition}
			The proof of   Proposition \ref{pr:movas} is a simple adaptation of the argument in \cite{movassagh_repeated2022}, and its proof is reported in Appendix \ref{app:duality}.
			When $\eta$ is a Dirac mass one has $D(\eta\|\pi)=\log n$, so using a reversed Pinsker inequality, Proposition \ref{pr:movas} yields a lower bound on mixing times:
			\begin{equation}\label{l1entbound}
			t_{\rm mix}(\varepsilon)\ge 	\left(1-\frac{\varepsilon \, n}{n-1}\right)t_{\rm ent} 
			\comma\qquad \varepsilon \in (0,1)\fstop
			\end{equation}

			Our first result shows that $t_{\rm ent}$ is the right   timescale for the decay of the random sequence $d_{\rm TV}(t)$.
			In particular, the 
			lower bound \eqref{l1entbound} is rather sharp.

			\begin{theorem}
			\label{pr:tent-tmix}
			For any block size distribution, 
	$t_{\rm mix}(\eps)\asymp t_{\rm ent}$,  for all $\eps\in (0,1)$.
			\end{theorem}
		
		\subsection{Cutoff phenomenon}
		We now turn our attention to 
		the occurrence of  cutoff, as well as to the identification of the cutoff window and cutoff profile.
		
		We start with a Block Average version of 
			the well-known \emph{product condition} for reversible Markov chains, stating that if the relaxation time is not much smaller than the mixing time then there is no cutoff, see, e.g., \cite[Section 18.3]{levin2017markov}. 
	\begin{proposition}[Entropic product condition]\label{pr:prod-cond}
				Cutoff as in \eqref{eq:cutoff-def} implies $t_{\rm ent}\gg t_{\rm rel}$.
			\end{proposition}
			A natural question is whether this entropic product condition is sufficient for cutoff in our context. In the usual Markov chain setting, this equivalence holds for specific classes of  Markov chains \cite{ding_lubetzky_peres_total_variation_2010,basu_hermon_peres_2017,salez2022universality}, but it is not true in general; see, e.g., \cite[Example 18.7]{levin2017markov}. 
Counterexamples to this equivalence for Block Averages, akin to the example in the aforementioned reference, are discussed in Section \ref{suse:alternative} below. Nevertheless, in Theorem \ref{th:cutoff}, we establish a sufficient condition for cutoff based on a strengthened entropic product condition.

			In order to state that, let us introduce one of the key quantities of our analysis.
			Let $Y$  be the size-biased version of $X$, so that the parameters $\mu=\E[\log Y]$ and $\sigma^2=\var(\log Y)$  
			satisfy 
			\begin{align}\label{eq:mu-sigma}
			\mu=\frac{\E[X\log X]}{\E[X]}\comma\qquad
			\sigma^2=\frac{\E[X(\log X)^2]}{\E[X]}-\mu^2\fstop
			\end{align}
			With this notation, the entropic time $t_{\rm ent}$ in \eqref{eq:t-ent0} reads 
			\begin{equation}\label{eq:t-ent}
				t_{\rm ent}= \frac{n\log n}{\E[X]\,\mu}\comma
			\end{equation}
			whereas the entropic product condition $t_{\rm ent}\gg t_{\rm rel}$ in Proposition \ref{pr:prod-cond}  becomes
				\begin{align}\label{eq:prod-condition2}
				\mu\ll\log(n)\fstop
			\end{align}
			Our next result shows that this condition is sufficient if combined with a  
			Weak Law of Large Numbers (WLLN) for the variables $\log Y$. We say that $\log Y$ satisfies the WLLN if 
				\begin{equation}\label{eq:WLLN}
				\lim_{n\to \infty}\Prob\bigg(\bigg|\frac1{\log n }{\sum_{s=1}^{\mu^{-1}\log n}\left(\log Y_s-\mu\right) }\bigg|>\delta \bigg)= 0\comma  \qquad  \delta >0\fstop
				\end{equation}

				\begin{theorem}[Cutoff]\label{th:cutoff}
				Assume the entropic product condition \eqref{eq:prod-condition2} and the WLLN  \eqref{eq:WLLN}. Then the Block Average process has  cutoff, 
				with $t_{\rm mix}(\eps)\sim t_{\rm ent}$,  $\eps \in (0,1)$.
			\end{theorem}
		In words, Theorem \ref{th:cutoff} states that the abrupt convergence to equilibrium for Block Averages at time $t_{\rm ent}$ is parallel to a {\rm WLLN} --- the simplest 
		concentration phenomenon in probability --- for some auxiliary i.i.d.\ triangular arrays of length $\mu^{-1}\log n$.   The next result (Theorem \ref{th:cutoff-window}) pushes this analogy further, establishing that the minimal assumption 
			\begin{equation}
			\label{eq:HP-noncutoff}
			\sigma^2\ll {\mu \log n}\comma
		\end{equation}
yields estimates on the size of the cutoff window.
		This condition is minimal because, in general, $\sigma^2 \le \mu \log n$ always holds true  (cf.\ \eqref{eq:mu-sigma} and $X\le n$). Note also that, by Chebyshev inequality,  \eqref{eq:HP-noncutoff} alone clearly implies \eqref{eq:WLLN}.
			\begin{theorem}[Cutoff window]\label{th:cutoff-window} Assume  \eqref{eq:prod-condition2} and \eqref{eq:HP-noncutoff}.
			Then, the quantity
			\begin{equation}\label{eq:tw}
				t_{\rm w}= \left(1+\frac{\sigma}{\mu}\right)\frac{n\sqrt{\log n}}{\E[X]\sqrt \mu}
			\end{equation}satisfies $t_{\rm w}\ll t_{\rm ent}$, and
				\begin{equation}
					t_{\rm mix}(\eps)-t_{\rm mix}(1-\eps)\lesssim t_{\rm w}  \comma\qquad \eps\in (0,\tfrac12)\fstop
			\end{equation}
		\end{theorem}
		\begin{remark}
			The assumptions in Theorem \ref{th:cutoff-window} may be equivalently restated as
		\begin{align}\label{eq:nec} 
			\frac{t_{\rm ent}}{t_{\rm rel}}\gg 1+\frac{\sigma^2}{\mu^2}\comma
		\end{align} that is, as a strengthened version of the entropic product condition in \eqref{eq:prod-condition2}. 
		\end{remark}
		
			To complete the picture, one is typically interested in capturing the second-order terms of the mixing times or, equivalently, determining the so-called cutoff profile.  It turns out that,  next to the assumptions of Theorem \ref{th:cutoff-window}, having a Central Limit Theorem (CLT) for $\log Y$ --- in place of the WLLN  in \eqref{eq:WLLN} --- suffices to grant a non-trivial Gaussian profile inside the window $t_{\rm w}$  introduced in Theorem \ref{th:cutoff-window}. To simplify the exposition we consider two separate scenarios, according to whether the size of the blocks $X$ is deterministic or not:
			\begin{itemize}
			\item scenario 1: \eqref{eq:prod-condition2} holds, and $\sigma\equiv 0$
			\item scenario 2: \eqref{eq:prod-condition2} and \eqref{eq:HP-noncutoff} hold, $\sigma\neq 0$ for all $n\in\bbN$, and the CLT holds: 
			\begin{equation}\label{eq:CLT}
				\lim_{n\to \infty} \Prob\bigg(\sqrt{\frac{\mu}{\log n}}\;\sum_{s=1}^{\mu^{-1}\log n}\frac{\log Y_s-\mu}{\sigma}> \beta\bigg) = \Phi(-\beta)\comma\qquad \beta \in \R\comma
			\end{equation}  
			where $\Phi$ is the CDF of a standard Gaussian. 
			\end{itemize} 
			 
		\begin{theorem}[Cutoff profile]
			\label{th:profile1} 
			Assume either scenario 1 or scenario 2. 
Letting $\rho=\frac{\sigma}{\mu}$, 
			\begin{equation}\label{eq:statement-profile}
				d_{\rm TV}(t_{\rm ent}+\beta t_{\rm w})-\Phi\bigg(-\beta\frac{1+ \varrho}{\sqrt{1+\varrho^2}} \bigg)\overset{\Prob}\longrightarrow 0\comma\qquad \beta \in \R\fstop
			\end{equation}
		\end{theorem}
		Here we use the notation $\overset{\Prob}\longrightarrow $ for convergence in probability. We also note that, as usual, all quantities in the above expressions (like $\mu,\sigma,\rho$) are allowed to vary with $n$, and we do not necessarily assume convergence of the sequence $\Phi\big(-\beta(1+\varrho)/\sqrt{1+\varrho^2}\big)$. 
			\begin{figure}[h]
			\includegraphics[width=7cm]{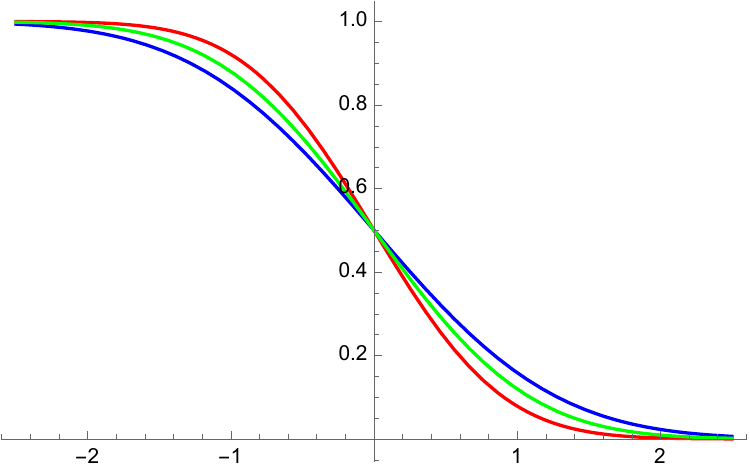}
			\caption{
				The cutoff profile in \eqref{eq:statement-profile} for three values of $\rho$, namely $\rho=0,1,5$ (in blue, red and green, respectively). }\label{fig:profile}
		\end{figure}
		\begin{remark}\label{rem:lindeberg}
			In view of Feller's theorem (see, e.g., \cite[Theorem 11.1.5]{athreya_lahiri_measure_2006}), simple manipulations show that, under \eqref{eq:prod-condition2} and \eqref{eq:HP-noncutoff},  the CLT assumption in \eqref{eq:CLT} is equivalent to the following Lindeberg condition for the triangular array $\log Y$:
			\begin{equation}\label{eq:lindeberg}
				\lim_{n\to\infty}\E\bigg[\bigg(\frac{\log Y -\mu}\sigma\bigg)^2\,\IND_{\left|\frac{\log Y-\mu}\sigma\right|>\delta\sqrt\frac{\log n}\mu}\bigg]=0\comma\qquad \delta>0\fstop
			\end{equation}
		\end{remark}

		As already anticipated in the introduction, the factor $\rho$ appearing in the cutoff profile
	\begin{equation}\label{eq:Psi}\beta\in \R\longmapsto\Psi_{\varrho}(\beta)= \Phi\bigg(-\beta\frac{1+\rho}{\sqrt{1+\rho^2}}\bigg)\in [0,1]\comma
		\end{equation} 
		accounts for the competition of two sources of fluctuations, one from the block sizes and one from the randomness of the block composition. In particular, when $\rho =0$ or $\infty$, only one of the two fluctuations asymptotically survives. Note that the factor $\rho$ crucially enters also in the definition of the cutoff window $t_{\rm w}$ as defined in \eqref{eq:tw}.
		
		It is worth to point out the symmetry of the profile with respect to $\rho$, that is,  \begin{equation}\Psi_{\varrho}(\beta)= \Psi_{\varrho^{-1}}(\beta)\comma\qquad \beta\in \R\fstop
			\end{equation} As a consequence of this symmetry and the monotonicity of $\rho\in [0,1]\mapsto(1+\rho)/{\sqrt{{1+\rho^2}}}$, the curve associated to $\rho =0$ is above (resp.\ below) the one associated to $\rho =1$ on the positive (resp.\ negative) half-line, whereas all other curves for $\rho \not\in\{0,1,\infty\}$ are sandwiched between these two; see Figure \ref{fig:profile}.

\subsection{Deterministic block sizes}
When the block size $X$ (and, thus, $\log Y$) is deterministic, 
Proposition \ref{pr:prod-cond} and Theorem \ref{th:profile1} (assuming scenario 1) yield a full characterization of cutoff -- including the cutoff profile -- in terms of the entropic product condition only, as summarized by the following corollary. Note that, in the deterministic case 
$X\equiv k$, one has
	\begin{equation}t_{\rm rel}= \frac{n-1}{k-1}\comma 	 \qquad t_{\rm ent}= \frac{n\log n}{k\log k}\comma\qquad t_{\rm w}=\frac{n\sqrt{\log n}}{k\sqrt{\log k}}\fstop \end{equation}
	\begin{corollary}[Cutoff  for deterministic sizes]\label{cor:cutoff-deterministic} Assume $X\equiv k$, for some sequence $k=k(n)$. Then, the following conditions are equivalent:
		\begin{itemize}
			\item$\log k \ll \log n$;
			\item cutoff occurs at time $t_{\rm ent}$ and, further,
			\begin{equation}d_{\rm TV}(t_{\rm ent}+\beta t_{\rm w})\overset{\Prob}\longrightarrow\Phi(-\beta)\comma\qquad \beta \in \R\fstop
			\end{equation}
		\end{itemize}   
		
	\end{corollary}
In particular, Corollary \ref{cor:cutoff-deterministic}  includes the main result of \cite{chatterjee2020phase}, which corresponds to $X\equiv 2$.

				We now turn to the analysis of mixing in the non-cutoff regime, in case of deterministic block sizes $X\equiv k$. By Corollary \ref{cor:cutoff-deterministic}, we, thus, consider the case where $k=n^\delta$, $\delta\in(0,1)$.

				As usual, the initial mass $\eta\in\Delta$ is assumed to be a Dirac delta at a single site $x_0\in V$.  The next result shows that after the first random time at which  the initial mass is updated, the total variation distance concentrates on the timescale $n/k$ around the survival function of a  Poisson process of unit rate.
				\begin{theorem}[Non-cutoff for deterministic sizes] 
			\label{th:polynomial}
			Fix $\delta\in(0,1)$ satisfying $\delta^{-1}\notin \N$, and take $X\equiv k$, with $k=n^\delta$. 
			Set  $\bar{t}(\beta)={\beta n/k}$, for $\beta \ge 0$, and define the random time
			\begin{equation}\label{eq:def-tau-start}
				\tau_{\rm start}=\inf\{t\ge 1:\; x_0\in A_t \}\fstop
			\end{equation}
			Then,	
			\begin{equation}\label{eq:conv-poly0}
				d_{\rm TV}(\tau_{\rm start}+\bar{t}(\beta))\overset{\Prob}{\longrightarrow}\Pr\(\tau_{\lfloor \delta^{-1}\rfloor}> \beta \)\comma\qquad \beta \ge 0\comma
			\end{equation}
			where $\tau_{j}$ is the $j$-th arrival time of a Poisson  process of unit rate. 
				\end{theorem}
			\begin{figure}[h]
			\includegraphics[width=6cm]{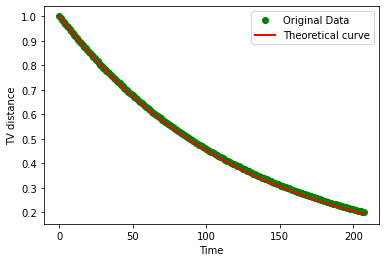}\qquad\qquad
			\includegraphics[width=6cm]{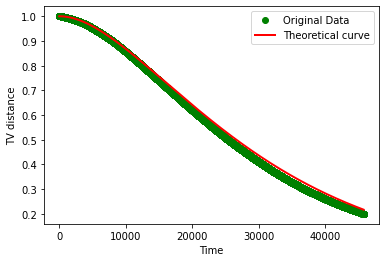}
			\caption{The plot represents the decay over time of $t\mapsto d_{\rm TV}(\tau_{\rm start}+t)$ in a single realization with $n=10^7$ and $k=n^{0.7}$ (left) and $k=n^{0.4}$ (right), up to the time $t$  such that $d_{\rm TV}(\tau_{\rm start}+t)=0.2$.}\label{fig:eps-large-1}
			\end{figure}
				\begin{remark}
				The assumption $\delta^{-1}\not\in\N$ is not  of purely technical nature, and the reason of its appearance  
				will be further discussed  in Section \ref{sec:deterministic}; see 
				Remark \ref{rem:eps-bar}.
			\end{remark}
			
				Since 
			$\tau_{\rm start}$ is geometrically distributed with mean $n/k$, we can restate \eqref{eq:conv-poly0} as a  convergence in distribution for the random variable $d_{\rm TV}(t)$, that is, 
				\begin{equation}\label{eq:conv-poly-dist}
				d_{\rm TV}(\bar t(\beta))\overset{\rm d}{\longrightarrow}\mathds{1}_{S\ge\beta}+\mathds{1}_{S<\beta}\Pr\(\tau_{\lfloor \delta^{-1}\rfloor}> \beta -S \)\comma\qquad \beta \ge 0\comma			\end{equation}
				where $S \overset{\rm d}={\rm Exp}(1)$ is a unit rate exponential random variable, independent of $\tau_{\lfloor \delta^{-1}\rfloor}$. 	Finally, taking expectations in \eqref{eq:conv-poly-dist}, we readily get
				\begin{equation}
				\E[d_{\rm TV}(\bar t(\beta))]\longrightarrow \Pr\big(\tau_{\lfloor \delta^{-1}\rfloor+1}> \beta\big)\comma\qquad \beta \ge 0\fstop
				\end{equation}

					In words, the  above results   
					state that the mixing of the process is articulated into two distinct phases. At first,  mixing needs to be triggered, by redistributing the unit mass initially placed in a single site. This first  time $\tau_{\rm start}$ is random, and approximately 
					exponentially distributed with mean $n/k$. After this triggering time, the distance starts to decay (again on timescales $n/k$), now in a deterministic way, following a Poissonian profile.  A similar two-phase decay has been recently observed on a related redistribution model \cite{avena2024mixing}.

\subsection{Outline of the paper}
The rest of the paper is organized as follows. Section \ref{sec:piles-general} is devoted to present the main tools of our analysis. There, we introduce the \emph{pile dynamics}, derive the distribution of the size of a typical pile and  its asymptotic behavior, and establish general upper and lower bounds on the total variation distance. In Section \ref{sec:proof-t-ent-t-mix} we prove Theorem \ref{pr:tent-tmix}, while in Section \ref{sec:proof-profile} we prove the cutoff results in Theorem \ref{th:cutoff} and Theorem \ref{th:cutoff-window}, and Section \ref{sec:profile} is devoted to the proof of Theorem \ref{th:profile1} on the cutoff profile. In Section \ref{sec:deterministic} we focus on the deterministic setup, proving Theorem \ref{th:polynomial}. Finally, in Section \ref{suse:alternative} we provide some examples showing the sharpness of the conditions in Theorems \ref{th:cutoff}, \ref{th:cutoff-window} and \ref{th:profile1}.

		\section{Piles \&  infinitesimal chunks}\label{sec:piles-general}
		In this section, we introduce the  central objects employed in the proofs of our main results. Specifically, we  present a richer representation of the block-average dynamics in terms of \textit{piles} and \textit{infinitesimal chunks}, rather than just the mass distribution $\eta_t$. 
	Following the evolution of sufficiently many of these piles will be key to measuring the flatness of the mass distribution.	
 The notion of piles and chunks is introduced in Section \ref{sec:piles}. 	In Sections \ref{sec:size-pile} and \ref{sec:size-pile-asymp}, we characterize the explicit distribution and corresponding asymptotic identities for the size of what we refer to as a \textquotedblleft typical pile\textquotedblright. These findings are exploited and refined in  later sections.

		\subsection{Pile-chunk dynamics}\label{sec:piles}
The original process is characterized by the variables $\eta_t(x)$, $t\ge 0, x\in V$, where $\eta_t(x)$ denotes the mass at vertex $x$ at time $t$, with the initial distribution $\eta_0$ given by the Dirac mass $\IND_{x_0}$ at some  $x_0\in V$. A finer description is obtained by keeping track of the distinct mass fragments produced by each averaging event, so that the mass $\eta_t(x)$ is given by the sum of the masses of all fragments present at $x$ at time $t$. These fragments are called \emph{piles}, and their evolution can be formally defined as follows.  
At time zero there is only one pile at $x_0\in V$, with mass equal to $\eta_0(x_0)=1$. Whenever there is an averaging event at the block $A\subset V$ with $|A|=k$, for every $x\in A$, each pile present at $x$ is independently split into $k$ labeled piles of equal mass, and the $k$ new piles are assigned to the sites $y\in A$ using a uniformly random permutation of $\{1,\dots,k\}$. In this way, at each time $t$, for each $x\in V$, if $\eta_t(x)>0$, then there is a finite number of piles at $x$, whose total mass equals $\eta_t(x)$, which we denote $(\zeta_t^{(1)}(x),\zeta_t^{(2)}(x),\ldots, \zeta_t^{(L)}(x))$, for some $L=L_t(x)\in \N$. 
	We refer to this new process $(\zeta_t)_{t\ge 0}$ as the \textit{pile dynamics}. Thus, letting $|\zeta_t^{(i)}(x)|\in [0,1]$ denote the \textit{size} of the pile $\zeta_t^{(i)}(x)$, at any time $t\ge 0$ one has 
	\begin{equation}
		\sum_i|\zeta_t^{(i)}(x)|=\eta_t(x)\comma\qquad t \ge 0\comma x \in V\,.
	\end{equation}

We further  
enrich the pile dynamics $(\zeta_t)_{t\ge 0}$ by adding a mark to one of the piles, which we interpret as the evolution  
of an \textquotedblleft infinitesimal chunk of mass\textquotedblright, 
and which we denote by $U$. 
The mark $U$ belongs initially to the unique pile at $x_0$. At each averaging event at the block $A\subset V$ with $|A|=k$, if the mark belongs to a pile sitting at one of the vertices of $A$, then the mark chooses uniformly at random one of the $k$ fragments into which its own pile is split and sticks to it until the next averaging event. When the mark belongs to a pile that is not involved in the averaging, the mark does not move.  Thus, at any time $t\ge 0$, the  infinitesimal chunk $U$ belongs to a single pile of the system, denoted, with slight abuse of notation, 
$\zeta_t(U)$, and has 
a \textit{position}, referred to as $e_t(U)\in V$,  namely the vertex at which the pile $\zeta_t(U)$ sits.
We also refer to $\zeta_t(U)$ as the marked pile, or the \emph{typical pile}. 

It follows from the above construction that the evolution of the position $e_t(U)$ of the infinitesimal chunk $U$, if we ignore all other variables, is simply the random walk with transition matrix $P_{\rm RW}$ introduced in \eqref{eq:duality}. Furthermore, if we condition on a realization of all the block updates, at each time $t\ge0$, by construction the probability that $U$ is at some $x\in V$ at time $t$ coincides with $\eta_t(x)$, i.e., $\Prob$-a.s., 
\begin{equation}\label{eq:duality1}
	{\mathsf P}_U(e_t(U)= x)=\eta_t(x)\comma\qquad t \ge 0\comma x \in V\comma
	\end{equation}
	where ${\mathsf P}_U$ denotes the quenched law of the evolution of the infinitesimal chunk $U$. 
	We shall write $\mathsf E_U$ for the corresponding expectation. Taking the $\E$-expectation on both sides of \eqref{eq:duality1}, that is averaging over all block updates, we recover the duality relation \eqref{eq:duality}.

	 The analysis of the evolution of the size $|\zeta_t(U)|$ of the marked pile will be another  key ingredient in our proofs. Specifically, we shall fully exploit the following identity
	 \begin{equation}\label{eq:pile-size-ex-mass}
	 	\Prob(|\zeta_t(U)|\in \cA) = \E\left[|\eta_t^\cA|\right]\comma\qquad t \ge 0\comma \cA\subset [0,1]\comma
	 \end{equation}
	 where $\eta_t^\cA$ is defined as the mass configuration obtained from $\eta_t$ by considering only the piles 
	 with size in $\cA$, and $|\eta_t^\cA|$ stands for its total mass $\le 1$. This is easily derived by taking the expectation with respect to $\Prob$ of the quenched identity
	 \begin{equation}\label{eq:pile-size-ex-mass-quenched}
	 	{\mathsf P}_U(|\zeta_t(U)|\in \cA) = |\eta_t^\cA|\comma \qquad t \ge 0\comma \cA\subset[0,1]\fstop
	 \end{equation}
	 Indeed, by the definition of the marked pile dynamics, the probability that $U$ is in a pile of a given size coincides with the total amount of mass in piles of that size. More generally, we  combine \eqref{eq:duality1} and \eqref{eq:pile-size-ex-mass-quenched} to obtain, e.g.,
	 \begin{equation}\label{eq:super-general}
	 	{\mathsf P}_U(e_t(U)=x\,,\, |\zeta_t(U)|\in \cA\,,\,\eta_t\in \cB) = \IND\{\eta_t\in \cB\}\,\eta_t^\cA(x)\comma
	 \end{equation}
	 where $t\ge 0$, $\cA\subset [0,1]$, $\cB\subset \Delta$, and $x\in V$. 
	 An  identity of this type allows us to follow selected portions of the mass evolving via block averages, depending on the size of the piles and their locations. As an example of its applications, see, e.g., \eqref{eq:fill-the-dots-identity} in Section \ref{suse:proof-tent-tmix}.

Moreover, we will gain refined information by following several infinitesimal chunks $U,U',U''$, etc., and their piles at the same time. Their dynamics is defined as follows. All chunks belong initially to the same pile at $x_0$, and when an average event occurs with block $A$, 
the chunks \textit{independently} -- albeit  simultaneously -- uniformly choose one among the newly generated sub-piles, and stick to it until the next averaging event.    It is worth to mention that, as soon as different infinitesimal chunks join different piles (and, thus, relocate to different sites), they will never sit on the same pile at later times. This, however, does not prevent them to have the same position later on. In formulas, this means that $\zeta_t(U)=\zeta_t(U')$ if and only if $\zeta_s(U)=\zeta_s(U')$ for all $s\le t$, whereas $e_t(U)=e_t(U')$ does not necessarily imply $\zeta_t(U)= \zeta_t(U')$.	We employ this approach in several places below; see, e.g., \eqref{eq:fill-the-dots}, Lemma \ref{lemma:pile-2}, Lemma \ref{lemma:high-mom}. In analogy with the notation used for a single mark, we will  write, e.g., $\mathsf P_{U,U',U''}$ and $\mathsf E_{U,U',U''}$ for the  joint quenched law and  expectation of the chunks $U,U',U''$.

 We remark that this pile-chunk representation may be viewed as a more flexible version of the \textquotedblleft particle model\textquotedblright\ introduced in \cite[Definition 2.5]{chatterjee2020phase}. Roughly speaking,  the authors in \cite{chatterjee2020phase} start by first discretizing the mass by  forming a finite number of  equal-sized
 piles, and then consider their evolution in a fashion similar to the one described above, with the important difference that piles which are either too small or that meet other piles coming from different locations get discarded. 
 These two operations are crucial when dealing with such a discretized version of the mass average dynamics. As we will show, our continuous pile-chunk dynamics avoids any discretization issue, and actually allows us to  track the pile dynamics even when meeting events between piles occur. This will play a prominent role in Section \ref{sec:deterministic}, where we consider block updates of such large size that cutoff is destroyed, and yet we are still able to identity the mixing profile.

\subsection{Distribution of the size of a typical pile}\label{sec:size-pile}
Consider an infinitesimal chunk $U$ as described in the previous section.
Since our blocks are chosen uniformly at random and independently, the size $|\zeta_t(U)|$ of its pile has a rather explicit distribution:
\begin{equation}\label{eq:size-pile-law}
	|\zeta_t(U)|\overset{\rm d}= \prod_{s=1}^t Z_s^{-1}\comma\qquad t \ge 1\comma
\end{equation}
where, letting $X_s$ denote the size of the block chosen at step $s$,  the random variables $(Z_s)_{s\ge 1}$ are i.i.d.\ with law given by
	\begin{equation}\label{eq:Z-law}
	\Prob(Z_s=1\mid X_s)=1-\frac{X_s}{n}\comma\qquad \Prob(Z_s=X_s\mid X_s)=\frac{X_s}{n}\fstop
\end{equation}
Indeed, $Z_s=1$ corresponds to the event \textquotedblleft the uniformly-chosen block of size $X_s$ did not involve the chunk $U$\textquotedblright, while $Z_s>1$ encodes the information that the update of size $X_s$ did hit the site $e_s(U)\in V$ and, thus, split the pile $\zeta_s(U)$ into $Z_s=X_s$ sub-piles. 
This readily implies \eqref{eq:size-pile-law} and \eqref{eq:Z-law}.

Thus, 
for all $t\ge 1$ and $\theta\in (0,1)$,
\begin{align}
	\Prob\big(|\zeta_t(U)|> \theta\big)  &= \Prob\bigg(\sum_{s=1}^t \log Z_s<  -\log\theta \bigg)\comma	
\end{align}
and  part of our work will go into deriving  
sharp asymptotics for the probability on the right-hand side. 
For this purpose, let $T=T(t)$ be the number of successful splittings, i.e., 	
\begin{equation}
	T=T(t):= \# \{s\le t: Z_s>1\}\in \{0,1,\ldots,t\}\comma
\end{equation}
which satisfies  
\begin{equation}\label{eq:def-T}
	T\overset{\rm d}= {\rm Bin}(t,\Prob(Z_1>1))\overset{\rm d}={\rm Bin}\bigg(t,\frac{\E[X]}{n}\bigg)\fstop
\end{equation}  Further, let us define  $Y_s:= Z_s\mid \{Z_s>1\}$, whose law is nothing but the size-biased distribution of $X_s$, that is, letting $p_X(k):= \Prob(X_s=k)$, we have
	\begin{equation}\label{eq:p_Y}
p_Y(k):= \Prob(Y_s=k)=	\Prob(Z_s=k\mid Z_s>1)=\frac{\Prob(X_s=k) \frac{k}{n}}{\sum_{k=2}^n \Prob(X_s=k)\frac{k}{n}}= \frac{k\,p_X(k)}{\E[X]}\comma
\end{equation}
for all $k\in \{2,\ldots,n\}$.
Hence, we obtain
\begin{align}\label{eq:size-typical-Z}
	\begin{aligned}
\Prob\bigg(\sum_{s=1}^t \log(Z_s)<  -\log \theta \bigg)&=	\Prob\bigg(\sum_{s=1}^T \log 	Z_s< {-\log \theta}\,\bigg|\, Z_s>1\,,\,\forall\, s\le T\bigg)\\
&= \Prob\bigg( \sum_{s=1}^T \log Y_s < {-\log \theta}\bigg)\fstop
\end{aligned}
\end{align}

By collecting the above steps and
observing that, conditional to $T=T(t)\in \{1,\ldots,t\}$, the variables $(Y_s)_{s\ge 1}$ are all	i.i.d.\ and distributed as in \eqref{eq:p_Y}, we just proved the following crucial result.
\begin{lemma}[Size of a typical pile]\label{lemma:size-pile}
For all $t\ge 1$ and $\theta\in (0,1)$,  we have
	\begin{equation}\label{eq:heur}
		\Prob\big(|\zeta_t(U)|>\theta\big)= \Prob\bigg(\sum_{s=1}^T \log Y_s <-\log \theta\bigg)\comma 
	\end{equation}
where $T\overset{\rm d}{=} {\rm Bin}(t,\frac{\E[X]}{n})$,  $(Y_s)_{s\ge 1}$ are i.i.d.\ with distribution $p_Y$ given in \eqref{eq:p_Y}, and all random 	variables are mutually independent.  
\end{lemma}

If one believes the heuristics
that the system is at equilibrium as soon as the typical pile has size $\approx\frac1n$, that is 
\begin{equation}\label{eq:guess}
	t_{\rm mix}\approx \inf\left\{t\ge0 \mid \Prob\(|\zeta_t(U)|\approx\tfrac1n\)\approx 1\right\}\comma
\end{equation}
 then Lemma \ref{lemma:size-pile} provides an efficient tool to guess the mixing time simply via a LLN for $\log Y$. Furthermore, in the cutoff regime, by controlling the fluctuations of $\log Y$ (that is,  having a CLT),   one may hope that this same heuristics could yield a candidate for the size of the  window, and, possibly even the corresponding profile. In the rest of the paper, we  turn the  heuristics in \eqref{eq:guess} into a rigorous framework. Moreover, we show how a Gaussian profile, with a variance depending on the interplay between the fluctuations of $T$ and  $\log Y$, arises in presence of cutoff.

\begin{remark}
Analogues of	Lemma \ref{lemma:size-pile},  and of the heuristics in \eqref{eq:guess}, appear in other works on random walks on random graphs \cite{ben-hamou_salez_2017,bordenave_random_2018,bordenave_cutoff2019}, as well as in \cite{chatterjee2020phase}, with some remarkable differences. In \cite{ben-hamou_salez_2017,bordenave_random_2018,bordenave_cutoff2019}, $Y_s^{-1}$ is the (out-)degree of a random vertex sampled with probability proportional to its (in-)degrees, but $T\equiv t$ deterministically. In \cite{chatterjee2020phase}, $T\overset{\rm d}{=}{\rm Bin}(t,\frac{2}{n})$, but  $\log Y\equiv \log 2$ deterministically.
\end{remark}

\subsection{Asymptotics for the size of a typical pile}\label{sec:size-pile-asymp}
A crucial step in the proofs of our main results consists in determining precise asymptotics for the tail probability of $|\zeta_t(U)|$, the typical size of a pile, appearing in Lemma \ref{lemma:size-pile}. This is the content of the following three lemmas (whose proof is deferred to the end of the section), dealing  with the assumptions in Theorem \ref{th:cutoff}, \ref{th:cutoff-window}, and \ref{th:profile1}, respectively. As already mentioned, the representation provided in Lemma \ref{lemma:size-pile} plays a crucial role here. In particular, we shall apply \eqref{eq:heur} with $\theta=e^\psi/n$ for some $\psi \ll \log n$, yielding
\begin{equation}\label{eq:main1}
		\Prob\bigg(|\zeta_t(U)|>\frac{e^\psi}n\bigg)= \Prob\bigg(\sum_{s=1}^T \log Y_s <\log n-\psi \bigg)\fstop
\end{equation}

The first lemma 
covers the setting in Theorem \ref{th:cutoff}.
	\begin{lemma}\label{lemma:main-WLLN}
		Assume  \eqref{eq:prod-condition2} and a {\rm WLLN} for $\log Y$ as in \eqref{eq:WLLN}.
		Then, for all  $0< \psi\ll \log n$, we have
		\begin{align}\label{eq:lemma-1-claim}
				\lim_{n\to \infty}	\Prob\bigg(|\zeta_{a t_{\rm ent}}(U)|> \frac{e^\psi}{n}\bigg)  = \begin{dcases}
					1 &\text{if}\ a <1\\
					0&\text{if}\ a >1 \fstop
				\end{dcases}  
		\end{align}
	\end{lemma}
	
	The second lemma deals with
	the setting in Theorem \ref{th:cutoff-window}. Let us introduce a convenient shorthand notation: recalling $t_{\rm ent}$ and $t_{\rm w}$ from \eqref{eq:t-ent} and  \eqref{eq:tw}, respectively, set
	\begin{equation}\label{eq:def-t-star-beta}
		t_*(\beta)=
		t_{\rm ent}+\beta t_{\rm w}\comma\qquad \beta \in \R\fstop
	\end{equation}

	\begin{lemma}\label{lemma:main-window}
		Assume \eqref{eq:prod-condition2} and \eqref{eq:HP-noncutoff}.
		Then, for all  $\gamma\in \R$ and $\psi= \gamma \sqrt{\mu\log n}$,	we have
			\begin{equation}\label{eq:main2}
			\lim_{\beta\to-\infty}\liminf_{n\to \infty}	\Prob\bigg(|\zeta_{t_*(\beta)}(U)|> \frac{e^\psi}{n}\bigg)  = 1\comma\quad
			\lim_{\beta\to\infty}\limsup_{n\to \infty}	\Prob\bigg(|\zeta_{t_*(\beta)}(U)|> \frac{e^\psi}{n}\bigg)  = 0 \comma
		\end{equation}
		where   $t_*(\beta)$ is defined in  \eqref{eq:def-t-star-beta}. 
	\end{lemma}

	The third lemma is concerned with the setting in Theorem \ref{th:profile1}. Here, we shall exploit Lemma \ref{lemma:main} which deals with a joint CLT for both variables $T$ and $\log Y$  on the right-hand side \eqref{eq:main1}. This suggests the convolution-type form of the limiting probabilities that we find in \eqref{eq:Xi} below. Observe that this limit profile coincides with the one in Theorem \ref{th:profile1} (see Remark \ref{rem:limit-Psi}).
		\begin{lemma}\label{lemma:main}
			In the setting of Theorem \ref{th:profile1}, 
			for all $\gamma\in \R$ and $\psi= \gamma \sqrt{\mu\log n}$,	we have
				\begin{align}
					&\lim_{n\to \infty}	\left|\Prob\bigg(|\zeta_{t_*(\beta)}(U)|> \frac{e^\psi}{n}\bigg)- \Xi_{\varrho}(\beta,\gamma) \right|=0\comma\qquad \beta \in \R\comma
				\end{align}
				where 	$t_*(\beta)$ is given in \eqref{eq:def-t-star-beta}, and
				\begin{equation}\label{eq:Xi} \Xi_{\varrho}(\beta,\gamma)=
					\int_{-\infty}^\infty \varphi(\alpha)\,	\Phi(-\tfrac1\varrho\left(\alpha+\beta(1+\varrho)+\gamma\right))\, \dd \alpha\fstop
				\end{equation}
				Here, $\varphi$ is the pdf of a standard Gaussian. As a convention,  we set $\Xi_{0}(\beta,\gamma)=\Phi(-\beta)$.
			\end{lemma}
			
			\begin{remark}\label{rem:limit-Psi} Although  $ \Xi_\rho(\beta,\gamma)$ and the cutoff profile in Theorem \ref{th:profile1} look different at first sight, they are essentially the same function. Indeed, a simple manipulation with convolution of two  Gaussian pdf's yields
				\begin{equation}
					\Xi_\rho(\beta,\gamma)= \Phi\bigg(-\frac{\beta\left(1+\rho\right)+\gamma}{\sqrt{1+\rho^2}}\bigg)\comma
				\end{equation}
				which, indeed, coincides with the cutoff profile found in Theorem \ref{th:profile1} after taking $\gamma\to 0$. 
			\end{remark}
			
We shall prove the three above lemmas in a unified way.
	\begin{proof}[Proof of Lemmas \ref{lemma:main-WLLN}, \ref{lemma:main-window} and \ref{lemma:main}]
	Recall from Lemma \ref{lemma:size-pile} that, for all $t\ge 0$, we have
	\begin{equation}\label{eq:tail}
		\Prob\bigg(|\zeta_t(U)|> \frac{e^\psi}{n}\bigg)
		= \sum_{m=0}^t \Prob\big(T=m\big)\, \Prob\bigg(\sum_{s=1}^m\log Y_s < \log n-\psi\bigg)\comma
	\end{equation}
	with $T\overset{\rm d}= {\rm Bin}(t,\frac{\E[X]}{n})$,  $Y \sim p_Y$ (cf.\  \eqref{eq:p_Y}), and all random variables are independent.
		Now, introduce the following piecewise constant functions of $\alpha \in \R$:
	\begin{equation}\label{eq:notationCLT}
			m(\alpha)\coloneqq \lfloor \E[T]+\alpha \sqrt{{\rm Var}(T)} \rfloor\comma
			\end{equation}
			\begin{equation}\label{eq:phi-n-f-n}
			\varphi_n(\alpha)\coloneqq \Prob(T=m(\alpha))\sqrt{{\rm Var}(T)}\comma\qquad
			f_n(\alpha)\coloneqq \Prob\left(\sum_{s=1}^{m(\alpha)}\log Y_s<\log n-\psi\right)\fstop
	\end{equation}
	With these definitions, the right-hand side of \eqref{eq:tail} becomes
	\begin{equation}
	\sum_{m=0}^t \Prob\big(T=m\big)\, \Prob\bigg(\sum_{s=1}^m\log Y_s < \log n-\psi\bigg)=	\int_\R \varphi_n(\alpha)\, f_n(\alpha)\, \dd \alpha \fstop
	\end{equation} 
	
	By condition \eqref{eq:prod-condition2}, we have  $\frac{\E[X]}{n}\ll 1$; hence, provided that 	$t\frac{\E[X]}{n}\gg 1$, 
	\begin{equation}\label{eq:var-exp-1}
	{\rm Var}(T)\sim	\E[T]= \frac{t\,\E[X]}{n}\gg 1	\fstop
	\end{equation}
	Moreover, since $T$ is a sum of integer-valued i.i.d.\ random variables, the following local CLT holds true (see, e.g., \cite[Theorem 7]{tao_blog_CLT}): recalling $\varphi_n$ in \eqref{eq:phi-n-f-n} and that $\varphi:\R\to \R_+$ denotes the pdf of a standard Gaussian, 
	\begin{equation}\label{eq:pointwise}
		\sup_{\alpha\in \R}|\varphi_n(\alpha)-\varphi(\alpha)|\ll 1\fstop
	\end{equation}
	Hence, since $\int_\R \varphi_n=\int_\R \varphi=1$, Scheffe's lemma, \eqref{eq:pointwise} and $|f_n|\le 1$ yield
	\begin{equation}\label{eq:scheffe}
	\int_\R \varphi_n(\alpha)\,f_n(\alpha)\, \dd \alpha\sim \int_\R \varphi(\alpha)\,f_n(\alpha)\, \dd \alpha\fstop
	\end{equation} 
	We are thus left to establishing limits for the right-hand side above. In particular, since $\varphi$ is integrable and smooth, and $|f_n|\le 1$, it suffices to determine $\liminf_{n\to \infty}$ and $\limsup_{n\to \infty}$, for a.e.\  $\alpha \in \R$, of
		\begin{equation}\label{eq:fn}
		f_n(\alpha) = \Prob\bigg(\sum_{s=1}^{m(\alpha)}(\log Y_s-\mu)<\log n-\mu\, m(\alpha)-\psi\bigg) \fstop
	\end{equation}  We divide the rest of the proof into three cases, one for  each lemma. We emphasize that, for all three definitions of $t$, the condition in \eqref{eq:var-exp-1} (and, thus, \eqref{eq:scheffe}) is automatically verified.
	
	\smallskip \noindent
	\textit{Case of Lemma \ref{lemma:main-WLLN}.} Fix $t=at_{\rm ent}$ with $a >0$, and $0<\psi \ll \log n$. Then, the expression in \eqref{eq:fn} is
 is estimated, for all $n\in \N$ sufficiently large, $\alpha \in \R$,	 and  $C=C(a,\alpha)> 0$, by
	\begin{align}
		f_n(\alpha)&\ge \Prob\bigg(\sum_{s=1}^{m(\alpha)}(\log Y_s-\mu)<(1-a)\log n-C\sqrt{\log n}-\psi\bigg)\\
		f_n(\alpha)&\le 
			\Prob\bigg(\sum_{s=1}^{m(\alpha)}(\log Y_s-\mu)<(1-a)\log n+ C\sqrt{\log n}-\psi\bigg)\fstop
	\end{align}
By the WLLN assumption in \eqref{eq:WLLN} and the fact that $(\log Y_s)_{s=1}^{m(\alpha)}$ are i.i.d.,  this proves that, for all $\alpha\in \R$,
\begin{equation}
	f_n(\alpha)\xrightarrow{n\to \infty}\begin{dcases}
		1 &\text{if}\ a <1\\
		0 &\text{if}\ a>1\comma
	\end{dcases}
\end{equation}
thus, establishing the desired claim for this case.

\smallskip \noindent
\textit{Case of Lemma \ref{lemma:main-window}.} Recall \eqref{eq:def-t-star-beta}. Fix $t=t_*(\beta)=t_{\rm ent}+\beta t_{\rm w}$ with $\beta \in \R$, and $\psi = \gamma\sqrt{\mu \log n}$ with $\gamma \in \R$. Thanks to the assumption \eqref{eq:HP-noncutoff}, the expression in \eqref{eq:fn} is estimated, for all  $\alpha,\beta,\gamma \in \R$,   $n\in \N$ sufficiently large and some $1\ll\xi\ll \sqrt{\mu\log n}$, by
\begin{align}
	f_n(\alpha)&\ge\Prob\bigg(\sum_{s=1}^{m(\alpha)}(\log Y_s-\mu)<-\left(\beta\left(1+\frac{\sigma}{\mu}\right)+\alpha+\gamma\right)\sqrt{\mu \log n}-\xi\bigg) \\
	f_n(\alpha)&\le\Prob\bigg(\sum_{s=1}^{m(\alpha)}(\log Y_s-\mu)<-\left(\beta\left(1+\frac{\sigma}{\mu}\right)+\alpha+	\gamma\right)\sqrt{\mu \log n}+\xi\bigg)\fstop
\end{align}
By Chebyshev inequality, we further get
\begin{align}
f_n(\alpha)&\ge 1-\frac8{\beta^2}\bigg(\frac{\frac\sigma\mu}{1+\frac\sigma\mu}\bigg)^2\ge 1-\frac8{\beta^2}\comma &&\beta <-2(\alpha+\gamma)\\
f_n(\alpha)&\le \frac8{\beta^2}\bigg(\frac{\frac\sigma\mu}{1+\frac\sigma\mu}\bigg)^2\le \frac8{\beta^2}\comma&&\beta >2 (\alpha+\gamma)\fstop
\end{align}
Hence, taking $\beta\to -\infty$ and $\beta\to \infty$, respectively, yields the desired claim for this case.

\smallskip \noindent
\textit{Case of Lemma \ref{lemma:main}.} Let	 
	\begin{equation}
		g_n(\alpha)\coloneqq \begin{dcases}
			\IND\{\alpha+\beta+\gamma<0\} &\text{if}\ \sigma = 0\\
			\Phi(-\mu/\sigma\left(\alpha+\beta\left(1+\sigma/\mu\right)+\gamma\right)) &\text{if}\ \sigma >0\fstop
		\end{dcases}
	\end{equation}
	By setting $t=t_*(\beta)$ and $\psi=\gamma \sqrt{\mu \log n}$ as in the previous case, the assumption in \eqref{eq:HP-noncutoff} ensures that, for all $\alpha,\beta,\gamma \in \R$ such that $\alpha+\beta+\gamma \neq 0$,
	\begin{equation}
		f_n(\alpha) \sim \begin{dcases}
			\IND\{\alpha+\beta+\gamma<0\} &\text{if}\ \sigma=0\\
			 \Prob\bigg(\sqrt{\frac\mu{\log n}}\sum_{s=1}^{\frac{\log n}{\mu}}\frac{\log Y_s-\mu}{\sigma}<-L(\alpha,\beta,\gamma)\bigg)&\text{if}\ \sigma> 0\comma
		\end{dcases}
	\end{equation}
	with
	\begin{equation}
	 L(\alpha,\beta,\gamma)= \frac\mu\sigma\left(\beta\left(1+\frac{\sigma}{\mu}\right)+\alpha+\gamma\right)\fstop
	\end{equation}Hence, the CLT assumption in \eqref{eq:CLT} yields, for all $\alpha,\beta, \gamma \in \R$ satisfying $\alpha+\beta +\gamma \neq 0$,	 
	\begin{equation}
		f_n(\alpha)-g_n(\alpha)\xrightarrow{n\to \infty}0\fstop
	\end{equation}
	This concludes the proof of the third case.	
\end{proof}

	We turn to a discussion of general upper and lower bounds for $\E[d_{\rm TV}(t)]$,
	which will be used several times in our proofs. The upper bound is based on an idea from \cite{chatterjee2020phase}, while the lower bound uses explicitly 
the pile dynamics introduced in Section \ref{sec:piles-general}.
		\subsection{A general upper bound for $\E[d_{\rm TV}(t)]$}\label{suse:general-ub}
	Given a fixed time $t$,  we consider a sub-probability distribution $\tilde \eta_t\le \eta_t$, obtained by removing some of the piles from the current configuration $\eta_t$. The precise definition of $\tilde \eta_t$ will be specified in the applications and it is not important at this stage.
Next, we let the system evolve from time $t$ to time $t+s$, for some fixed time $s$, with initial  distribution $\tilde \eta_t$. 
This can be coupled to $\eta_{t+s}$ by using the same updates in the time interval $[t,t+s]$. 
By the triangle inequality and the fact that $|\eta_r|=1$ for all $r\ge 0$, we get
		\begin{align}\label{eq:ub-step1}
		 \|\eta_{t+s}-\pi\|_{\rm TV}
			&\le \frac12\left\{\sum_{x\in V}\left|\eta_{t+s}(x)-\frac1n- \left(\tilde \eta_{t+s}(x)-\frac{|\tilde \eta_t|}{n}\right)\right| + \sum_{x\in V}\left|\tilde \eta_{t+s}(x)-\frac{|\tilde \eta_t|}{n}\right|\right\}\\
			&= 1-|\tilde \eta_t| + |\tilde \eta_t|\left\|\frac{\tilde \eta_{t+s}}{|\tilde \eta_t|}-\pi\right\|_{\rm TV}
			\le 1-|\tilde \eta_t|+\frac{|\tilde \eta_t|}2\left\|\frac{\tilde \eta_{t+s}}{|\tilde \eta_t|\pi}-1\right\|_2\fstop
		\end{align}
	Taking expectation on both sides,  letting $\cF_t$ denote the $\sigma$-algebra up to time $t$, Jensen inequalities and the $L^2$-bound in Proposition \ref{prop:L2}  yield
		\begin{align}
			\E[\|\eta_{t+s}-\pi\|_{\rm TV}]&\le 1-\E[|\tilde \eta_t|]+\frac12\, \E\bigg[|\tilde \eta_t|\E\bigg[\bigg\|\frac{\tilde \eta_{t+s}}{|\tilde \eta_t|\pi}-1\bigg\|_2\bigg|\cF_t\bigg]\bigg]\\
			&\le 1-\E[|\tilde \eta_t|] + \frac12 \bigg(1-\frac1{t_{\rm rel}}\bigg)^{s/2}\E\bigg[|\tilde \eta_t|\bigg\|\frac{\tilde \eta_t}{|\tilde \eta_t|\pi}-1\bigg\|_2\bigg]\\
			&\le 1-\E[|\tilde \eta_t|]+\frac{e^{-s/(2t_{\rm rel})}} 2\, n\, \E\Big[\max_{x\in V}\tilde \eta_t(x)\Big]\ .
		\end{align}
				For later reference, we rephrase the above argument in terms of a lemma.		
		\begin{lemma}\label{lemma:triangle-ineq-ub} 
			For all initial sub-distributions $0\le \xi \le \eta_0$ and times $s\ge 0$, we have
			\begin{equation}
				\E[\|\eta_s-\pi\|_{\rm TV}]\le 1-\E[|\xi|] + \exp\left(-s/2t_{\rm rel}\right) n\, \E\left[\|\xi\|_\infty\right]\ .
			\end{equation}
			
		\end{lemma}

	\subsection{A general lower bound for $\E[d_{\rm TV}(t)]$}\label{suse:general-lb}
	Consider the process $(\eta_t)_{t\ge 0}$ and the corresponding pile dynamics $(\zeta_t)_{t\ge 0}$. For a fixed time  $t\ge 0$, $a>1$, let  $w_t=w_t^a$ denote the mass distribution 
	obtained after the removal of all piles having mass smaller than $\frac{a}{n}$.

\begin{lemma}\label{lem:glb}	For all $t\ge 0$, $a>1$, and $\eps \in (0,1)$, almost surely
	\begin{equation}
		\|\eta_t-\pi\|_{\rm TV}
\ge
		\IND\{|w_t|>1-\eps\} \left(1-\eps -a^{-1}\right)\fstop
		\label{eq:lb}
	\end{equation}
Moreover, 
\begin{equation}\label{eq:general-lb}
	\E[\|\eta_t-\pi\|_{\rm TV}]\ge \big(1-\eps^{-1}\Prob(|\zeta_t(U)|<\tfrac{a}n)\big)(1-\eps-a^{-1})\fstop
\end{equation}
\end{lemma}
\begin{proof}
Note that 
\begin{equation}
	\|\eta_t-\pi\|_{\rm TV}=\sum_{x\in V}[\eta_t(x)-\tfrac1n]_+\ge \sum_{x\in V}[w_t(x)-\tfrac1n]_+=|w_t|
	-\tfrac1n|\{y\mid w_t(y)\ge\tfrac an \}.
\end{equation}
On the other hand,
$
|\{y\mid w_t(y)\ge \tfrac an \} |\le \frac na$. 
Thus,	restricting to the event $|w_t|>1-\eps$ yields the claim \eqref{eq:lb}.
	Next, by  Markov inequality, and using 
	\eqref{eq:pile-size-ex-mass},
\begin{equation}
	\Prob(|w_t|>1-\eps) \ge 1-\frac{1-\E[|w_t|]}\eps =1-\frac{\Prob(|\zeta_t(U)|<\tfrac{a}n)}\eps\,.
\end{equation}
Taking expectations on both sides of \eqref{eq:lb} we obtain \eqref{eq:general-lb}.
\end{proof}

		\section{Asymptotic equivalence of $t_{\rm ent}$ and $t_{\rm mix}$ }\label{sec:proof-t-ent-t-mix}
In this section we prove Theorem \ref{pr:tent-tmix}.		\subsection{Proof of lower bound in Theorem \ref{pr:tent-tmix}}
Recall from \eqref{eq:t-ent} the definition of $t_{\rm ent}$, and let $\eps \in (0,1)$ and $\delta>0$.  By applying the general lower bound in \eqref{eq:general-lb} with $t=\delta t_{\rm ent}$ and $a=\eps^{-1}>1$, we have
\begin{equation}
	\E[\|\eta_t-\pi\|_{\rm TV}]\ge\big(1-\eps^{-1}\Prob(|\zeta_t(U)|<\tfrac1{\eps n})\big) \left(1-2\eps\right)\fstop
\end{equation}
To bound the above probability, we employ the result in Lemma \ref{lemma:size-pile}. Let  $T\sim {\rm Bin}(t,\E[X]/n)$ and, independently,  let $(Y_s)_{s\ge 1}$ be i.i.d.\ with $Y_s\sim p_Y$. 
For all $n\in \N$ sufficiently large, 
\begin{align}
\Prob(|\zeta_t(U)|<\tfrac1{\eps n})&=\Prob\bigg(\sum_{s=1}^T \log Y_s >\log n +\log \eps\bigg)\le \sum_{m=0}^t \Prob(T=m)\, \frac{m\E[\log Y]}{\log n+\log \eps}
\\
&= \frac{\E[T]\E[\log Y]}{\log n+\log \eps}
= \delta\, \frac{\log n}{\log n+\log \eps}\comma
\end{align}
where we used Markov inequality and the definition $t=\delta t_{\rm ent}$.
Therefore, the desired result follows by taking first $n\to \infty$, then $\delta\to 0$, and finally $\eps \to 0$.	

	\subsection{Proof of upper bound in Theorem \ref{pr:tent-tmix} }\label{suse:proof-tent-tmix}

		Fix $t=a t_{\rm ent}$, $w=\sqrt{\frac{\log n}{\mu}}\in [1,\sqrt{\log n}]$, and define 
		$$t_\star(a,b)=t+b t_{\rm rel}w=a t_{\rm ent} + b t_{\rm rel}w
		$$  
		for some $b>0$. Observe that, since $\mu\le \log(n)$,
		$$\frac{n\log(n)}{\mu\E[X]}=t_{\rm ent}\ge \frac 12 t_{\rm rel} w=\frac12\frac{n-1}{\E[X]-1}\sqrt{\frac{\log(n)}{\mu}}\comma$$
		and therefore, for all constant $a,b>0$, $t_{\star}(a,b)\asymp t_{\rm ent}$.
		
		Let $\eta_t^*$ be the restriction of $\eta_t$ to those piles of mass smaller than $\frac{1}n$. Define $S(k):=\sum_{s=1}^{T} \log(Y_s)$, $k\in\bbN$. By Lemma \ref{lemma:size-pile} for all $a>0$ large enough,
		\begin{align}
			\E[|\eta_t^*|]&=\Prob(|\zeta_t(U)|<\tfrac{1}n)=\Prob\(S(T)>\log(n) \)
			\\
			&\ge \Prob\(S(\lceil\tfrac{a\log(n)}{2\mu}\rceil)>\log(n)\comma T>\tfrac{a\log(n)}{2\mu} \)\\
			&=\Prob\(S(\lceil\tfrac{a\log(n)}{2\mu}\rceil)>\log(n)\)\,\Prob\( T>\tfrac{a\log(n)}{2\mu} \)\fstop
			\label{eq:1bis}
		\end{align}  
Note that the second probability on the right-hand side of \eqref{eq:1bis} goes to $1$ provided $a>0$ is large enough. 
Similarly, the first probability on the right-hand side of \eqref{eq:1bis} 
satisfies, for any $a>6$,
	\begin{align}
\Prob\(S(\lceil\tfrac{a\log(n)}{2\mu}\rceil)>\log(n)\)
&\ge1- \Prob\(\left|S(\lceil\tfrac{a\log(n)}{2\mu}\rceil)-\tfrac a2\log(n)\right|\ge \tfrac{a}{3}\log(n)\)\\
&\ge1-\frac{9\sigma^2}{2a\mu\log(n)} \fstop
	\end{align}
Since $\sigma^2\le \mu\log(n)$, for any $\varepsilon>0$, taking $a=a(\varepsilon)$ large enough,
\begin{equation}\label{eq:1}
		\E[|\eta_t^*|]=\Prob(|\zeta_t(U)|<\tfrac{1}n)\ge 1-\varepsilon\fstop
\end{equation}  
		Next,  for some $c>0$, define 
		\begin{equation}\tilde{\eta}_t(x):=\eta_t^*(x)\mathds{1}\{\eta_t^*(x)\le \tfrac{e^{cw}}{n}\}\ ,
		\end{equation}
		and note that  with this definition of $\tilde\eta_t$, by Lemma \ref{lemma:triangle-ineq-ub} we have
				\begin{align}\begin{aligned}
			\E[\|\eta_{t_\star(a,b)}-\pi\|_{\rm TV}]&\le 1- \E[|\tilde \eta_t|] + \exp\left(-b w/2+c w\right)\\
			&= 1-\E[|\eta_t^*|]+\E[|\eta_t^*(B)|]+\exp\left(-b \tfrac w2+c w\right) \comma
			\end{aligned}
			\label{eq:ub}
		\end{align}
		where  $B:=\{x\in V\mid \eta_t^*(x)>\frac{e^{cw}}{n}\}$, so that
		\begin{equation}
			\E[|\tilde\eta_t|]=\sum_{x\in V} \E\big[\eta_t^*(x)\mathds{1}\{\eta_t^*(x)\le \tfrac{e^{cw}}{n}\}  \big]=\E[|\eta_t^*|]-\E[\eta_t^*(B)]\fstop
		\end{equation}
	By \eqref{eq:1} and \eqref{eq:ub}, it suffices to prove $\E[\eta_t^*(B)]\to 0$ after taking the limits $n\to \infty$ and, then,  $c\to \infty$. To this end, we crucially exploit the identity \eqref{eq:super-general} for the pile-chunk dynamics. We have
\begin{equation}\label{eq:fill-the-dots-identity}
			\begin{aligned}
				\eta_t^*(B)&=\mathsf P_U(e_t(U)\in B\,,|\zeta_t(U)|<\tfrac{1}n)\\
				&=\mathsf P_U(\eta_t^*(e_t(U))> \tfrac{e^{cw}}{n}\,,\,|\zeta_t(U)|<\tfrac{1}n)\\
				&=\mathsf E_U\big[\IND\{\eta_t^*(e_t(U))>\tfrac{e^{cw}}n\,,\,|\zeta_t(U)|< \tfrac{1}n\}\big]\\
				&= \sum_{x\in V} \mathsf E_U\big[\IND\{\eta_t^*(x)>\tfrac{e^{cw}}n\,,\, e_t(U)=x\,,\, |\zeta_t(U)|< \tfrac{1}n\}\big]\\
				&= \sum_{x\in V}\IND\{\eta_t^*(x)>\tfrac{e^{cw}}{n}\}\,\mathsf P_U\big(e_t(U)=x\,,\, |\zeta_t(U)|< \tfrac{1}n\big)\fstop
			\end{aligned}
		\end{equation}
 By taking expectation in \eqref{eq:fill-the-dots-identity}, applying Markov inequality, and reasoning as in Section \ref{sec:piles}, we then deduce
		\begin{equation}\label{eq:fill-the-dots}
			\begin{aligned}
				\E[\eta_t^*(B)]
				&\le \frac{n}{e^{cw}} \sum_{x\in V}\E\big[\eta_t^*(x)\,\mathsf P_U\big(e_t(U)=x\,,\, |\zeta_t(U)|< \tfrac{1}n\big)\big]
				\\
				&=  \frac{n}{e^{cw}}\sum_{x\in V}\E\big[\mathsf P_{U,U'}(e_t(U)=e_t(U')=x\,,\,|\zeta_t(U)|< \tfrac{1}n\,,\,|\zeta_t(U')|< \tfrac{1}n)\big]\\
				&= \frac{n}{e^{cw}}\,\Prob\big(e_t(U)=e_t(U')\,,\, |\zeta_t(U)|< \tfrac{1}n\,,\, |\zeta_t(U')|<\tfrac{1}n\big)\fstop
			\end{aligned}
		\end{equation}
			We now estimate the above probability.
		\begin{lemma}\label{lemma:pile-2}
			For all $t\ge 0$, $\theta\in (0,1)$,  and $n\in \N$ sufficiently large, we have
			\begin{equation}
				\Prob\big(e_t(U)=e_t(U')\,,\, |\zeta_t(U)|< \theta\,,\, |\zeta_t(U')|<\theta\big)\le  \theta+\frac1n\fstop
			\end{equation}
		\end{lemma} 
		\begin{proof}
			First, observe that, under the event $\{|\zeta_t(U)|<\theta\,,\, |\zeta_t(U')|<\theta\}$,  the probability that $U$ and $U'$ sit on the same pile is at most $\theta$. Indeed, 
			\begin{align}
				&\Prob\left(\zeta_t(U)=\zeta_t(U')\,,\,
				|\zeta_t(U)|<\theta\,,\, |\zeta_t(U')|<\theta \right)\\
				& = 	\Prob\left(\zeta_t(U)=\zeta_t(U')\,,\,
				|\zeta_t(U)|<\theta \right)\\
				&= \E[\IND\{|\zeta_t(U)|<\theta\}\Prob(\zeta_t(U)=\zeta_t(U') \mid |\zeta_t(U)|)]\\
				&\le \theta\, \Prob(|\zeta_t(U)|<\theta)\le \theta
				 \fstop
			\end{align} Hence, up to errors of size $\theta$ and recalling that initially we have $\{e_0(U)=e_0(U')\}$, the probability that the event $\{e_t(U)=e_t(U')\}$ occurs is given by the probability of the following event: there exists a time $r\in [1,t)$ at which the two chunks go to different sites, there exists a time $s\in [r+1,t]$ in which they sit on the same vertex for the first time after $r$, and then we have $\{e_{\ell}(U)=e_{\ell}(U')\,,\,\forall \ell\in[s,t] \}$. In formulas, this means
		\begin{equation}
				\Prob\big(e_t(U)=e_t(U')\,,\, |\zeta_t(U)|< \theta\,,\, |\zeta_t(U')|< \theta\big)
		 \le  \theta+	\Prob(\exists\, s\le t\ \text{s.t.}\ \cM_{s-1}^\complement\cap\cM_{s,t})\comma
			\label{eq:first-bound}			
		\end{equation}
		where we define the events
		$\cM_r\coloneqq \{e_{r}(U)= e_{r}(U') \}$, and,  for $s< t$, 
		$\cM_{s,t}\coloneqq \cap_{r=s}^t\, \cM_r$. We now estimate the probability on the right-hand side of \eqref{eq:first-bound}. A union bound yields
		\begin{equation}\label{eq:display-1}
			\Prob(\exists\, s\le t\ \text{s.t.}\ \cM_{s-1}^\complement\cap\cM_{s,t})\le  \sum_{s\le t}\Prob(\cM_{s-1}^\complement \cap\cM_{s,t})\ .
		\end{equation}
	Furthermore, by the Markov property, 
		\begin{align}
			\Prob(\cM_{s-1}^\complement\cap\cM_{s,t} )&= \Prob(\cM_{s+1,t}\mid \cM_{s})\,\Prob(\cM_{s}\mid \cM_{s-1}^\complement)\, \Prob(\cM_{s-1}^\complement)\\
			&\le \Prob(\cM_{s+1,t}\mid \cM_{s})\,\Prob(\cM_{s}\mid \cM_{s-1}^\complement)\\
			&= \Prob(\cM_{s+1,t}\mid \cM_{s}) \sum_{k=2}^{n}p_X(k) \frac{k(k-1)}{n(n-1)} \frac1k\\
			&= \Prob(\cM_{s+1,t}\mid \cM_{s})\, \frac1n\frac1{t_{\rm rel}} \ ,\qquad s<t\ ,
			\label{eq:display-2}
		\end{align}
		where the second-to-last step follows from the fact that the probability of selecting a $k$-block containing the two vertices $e_{s-1}(U)\neq e_{s-1}(U')$ is $\binom{n-2}{k-2}/\binom{n}{k}=\frac{k\left(k-1\right)}{n\left(n-1\right)}$, while $\frac1k$ is the probability of having them sitting on the same site after they have been both involved in a $k$-block update. Analogously, we have
		\begin{equation}\label{eq:display-3}
			\Prob\left(\cM_{t-1}^\complement\cap \cM_{t}\right)\le \frac1n\frac1{t_{\rm rel}}\ .
		\end{equation}
		By the Markov property, stationarity and $e_0(U)=e_0(U')$, \begin{equation}\label{eq:display-4}\Prob(\cM_{s+1,t}\mid \cM_{s})=  \prod_{u=s}^{t-1} \Prob(\cM_{u+1}\mid\cM_{u})=  \left(\Prob(\cM_{1})\right)^{t-s}\fstop
		\end{equation} Moreover, 
		\begin{equation}
			\Prob(\cM_{1})=\sum_{k=2}^np_X(k)\left\{\left(1-\frac{k}{n}\right)+\left(\frac{k}{n}\frac{1}{k}\right)\right\}= 1-\frac{n}{n-1}\frac1{t_{\rm rel}}\le 1-\frac1{t_{\rm rel}}\fstop
		\end{equation}
		Summarizing, 
		\begin{align}
			\Prob(\exists\, s\le t\ \text{s.t.}\ \cM_{s-1}^\complement\cap\cM_{s,t})&\le \frac1n\frac1{t_{\rm rel}} \sum_{s\le t} \left(1-\frac1{t_{\rm rel}}\right)^{t-s}\le \frac1n\ .
		\end{align}
		This and  \eqref{eq:first-bound} yield the desired result.
		\end{proof}
		By applying Lemma \ref{lemma:pile-2} to \eqref{eq:fill-the-dots}, we obtain
		\begin{equation}
			\E[\eta_t^*(B)]\le 2e^{-cw}\comma
		\end{equation}
		and, thus,  the upper bound in Theorem \ref{pr:tent-tmix} follows after choosing $b>c>0$ sufficiently large.

					\section{Cutoff}\label{sec:proof-profile}
					
						In this section we first establish Proposition \ref{pr:prod-cond}, and then prove Theorem \ref{th:cutoff} and Theorem \ref{th:cutoff-window}.

					\subsection{The entropic product condition is necessary for cutoff}\label{sec:product-cond}
			
				By Jensen inequality, 
				\begin{equation}\label{eq:jensen-lb}
					\E[d_{\rm TV}(t)]=\E\big[\big\|\eta_t-\pi\big\|_{\rm TV}\big]\ge \big\|\E[\eta_t]-\pi\big\|_{\rm TV}\ge \frac12\bigg(1-\frac1{t_{\rm rel}}\bigg)^t\ ,
				\end{equation}
				where we use the duality \eqref{eq:duality} and 
				a well-known lower bound  for the distance-to-equilibrium of ${\rm RW}$; see, e.g., \cite[Eq.\ (12.15)]{levin2017markov}. 
				Observe that \eqref{eq:jensen-lb} ensures that
				\begin{equation}\label{eq:ratio}
					\frac{t_{\rm mix}(\varepsilon)}{t_{\rm mix}(1/4)} \ge \frac{t_{\rm rel}-1}{t_{\rm mix}(1/4)}\log\bigg(\frac1{2\varepsilon}\bigg)\fstop
				\end{equation}
				Therefore, if 
$t_{\rm ent}\lesssim t_{\rm rel}$, then Theorem \ref{pr:tent-tmix} and  \eqref{eq:ratio} imply that 
the ratio $t_{\rm mix}(\varepsilon)/t_{\rm mix}(1/4)$ 
diverges to $\infty$, uniformly over $n$, as $\varepsilon \to 0$. In other words, the condition $t_{\rm rel}\ll t_{\rm ent}$ is necessary for the sequence $t\mapsto\E[d_{\rm TV}(t)]$ to exhibit cutoff. This proves Proposition \ref{pr:prod-cond}.

				\subsection{Proof of lower bounds in Theorems \ref{th:cutoff} and \ref{th:cutoff-window}}\label{suse:lb-cutoff}
					Fix $\eps \in (0,1)$.
					Assume \eqref{eq:prod-condition2} and \eqref{eq:WLLN}, and apply the estimate in \eqref{eq:general-lb} with $t=(1-\eps)t_{\rm ent}$ and $a= e^\psi$, for some $1\ll \psi\ll \log n$.
					Then, Lemma \ref{lemma:main-WLLN} ensures that
					\begin{equation}
						\lim_{n\to\infty}\Prob(|\zeta_t(U)|<\tfrac{e^\psi}n)=1\comma
					\end{equation}
					hence, concluding the proof of the lower bound of Theorem \ref{th:cutoff}.

					We proceed similarly, for the lower bound in Theorem \ref{th:cutoff-window}. Use the notation in Lemma \ref{lemma:main} and
					assume \eqref{eq:prod-condition2} and \eqref{eq:HP-noncutoff}. Applying the estimate in \eqref{eq:general-lb} with $t=t_*(\beta):= t_{\rm ent}+\beta t_{\rm w}$, $a=e^\psi$, where $\psi:=\gamma \sqrt{\mu \log n}$, for some $\gamma > 0$.	(Note that  conditions \eqref{eq:prod-condition2}--\eqref{eq:HP-noncutoff} ensure that $t_{\rm w}\ll t_{\rm ent}$; moreover, we have $\psi\gg 1$.) Then, Lemma \ref{lemma:main-window} ensures that, for all $\gamma >0$, 
					\begin{equation}
						\lim_{\beta\to-\infty}\lim_{n\to\infty}\Prob(|\zeta_t(U)|<\tfrac{e^\psi}n)=1\fstop
					\end{equation}

		\subsection{Proof of upper bounds in Theorems \ref{th:cutoff} and \ref{th:cutoff-window}}\label{suse:ub-cutoff}
	
				Fix any $t,s\ge 0$, $\psi>0$ and, given the block-averaging  and the corresponding pile dynamics, split $\eta_t$ into two parts: $\eta_t^+$ denotes the mass consisting of piles of size $\ge \frac{e^\psi}n$, $\eta_t^-$ the rest, so that $\eta_t=\eta_t^+ +\eta_t^-$.
	We now consider the coupled processes $(\eta_{t+s}^+,\eta_{t+s}^-)_{s\ge 0}$ starting at $\eta_t^+$ and $\eta_t^-$, respectively, and evolving according to the same block averages. By the triangle inequality, we get, for all $s\ge 0$,
		\begin{align}
			\E[d_{\rm TV}(t+s)]&\le \E\bigg[|\eta_t^+|\,\bigg\|\frac{\eta_{t+s}^+}{|\eta_t^+|}-\pi\bigg\|_{\rm TV}\bigg] + \E\bigg[|\eta_t^-|\,\bigg\|\frac{\eta_{t+s}^-}{|\eta_t^-|}-\pi\bigg\|_{\rm TV}\bigg]\\
			&\le \E[|\eta_t^+|] + \E\bigg[|\eta_t^-|\,\bigg\|\frac{\eta_{t+s}^-}{|\eta_t^-|}-\pi\bigg\|_{\rm TV}\bigg]\\
			&= \Prob\bigg(|\zeta_t(U)|>\frac{e^\psi}n\bigg) +  \E\bigg[|\eta_t^-|\,\bigg\|\frac{\eta_{t+s}^-}{|\eta_t^-|}-\pi\bigg\|_{\rm TV}\bigg]\fstop
		\end{align} 
		
			Now, let us consider the sub-configuration $\eta_{t+s/2}^*$ of $\eta_{t+s/2}^-$ given by
		\begin{equation}
			\eta_{t+s/2}^*(x):= \eta_{t+s/2}^-(x)\, \IND\{\eta_{t+s/2}^-(x)\le \tfrac{e^{3\psi}}n\}\comma\qquad x \in V\fstop
		\end{equation}
		Hence, by the Markov property and the upper bound in Lemma \ref{lemma:triangle-ineq-ub}, we obtain
		\begin{align}
			\E\bigg[|\eta_t^-|\,\bigg\|\frac{\eta_{t+s}^-}{|\eta_t^-|}-\pi\bigg\|_{\rm TV}\bigg]&\le \E\bigg[|\eta_t^-|\bigg(1-\frac{|\eta_{t+s/2}^*|}{|\eta_t^-|}\bigg)+|\eta_t^-| \exp\left(-s/4t_{\rm rel}+3\psi\right)\bigg] \\
			&\le\E[|\eta_t^-|]-\E[|\eta_{t+s/2}^*|]+\exp\left(-s/4t_{\rm rel}+3\psi\right)\\
			&= \E[|\eta_t^-|]-\E[|\eta_{t+s/2}^*|]+\exp\left(-\psi\right)\comma
		\end{align}
			where the last identity holds by choosing $s= 16t_{\rm rel}\psi\ge 0$. 
			
			Letting $B:=\{x\in V: \eta_{t+s/2}^-(x)>\frac{e^{3\psi}}{n}\}$ so that $\eta_{t+s/2}^-(B)=|\eta_t^-|-|\eta_{t+s/2}^*|$, we obtain, by arguing as in \eqref{eq:fill-the-dots}, 
			\begin{align}
				\E[\eta_{t+s/2}^-(B)]&\le \frac{n}{e^{3\psi}}\, \Prob(e_{t+s/2}(U)=e_{t+s/2}(U')\,,\, |\zeta_t(U)|<\tfrac{e^\psi}{n}\,,\, |\zeta_t(U')|<\tfrac{e^\psi}n)\\
				&\le 
				\frac{n}{e^{3\psi}}\, \Prob(e_{t+s/2}(U)=e_{t+s/2}(U')\,,\, |\zeta_{t+s/2}(U)|<\tfrac{e^\psi}{n}\,,\, |\zeta_{t+s/2}(U')|<\tfrac{e^\psi}n)\comma
			\end{align}
			where the last step holds true because the pile sizes are a.s.\ non-increasing. By applying Lemma \ref{lemma:pile-2}, we further get
			\begin{equation}
				\E[\eta_{t+s/2}^-(B)]\le\frac{n}{e^{3\psi}}\left( \frac{e^\psi}{n}+\frac1n\right)\le 2e^{-2\psi}\fstop
			\end{equation} 
			Putting all things together, we get 
				\begin{equation}\label{eq:ub-cutoff}
					\E[d_{\rm TV}(t+s)]\le \Prob\bigg(|\zeta_t(U)|>\frac{e^\psi}n\bigg) + 2e^{-2\psi}+e^{-\psi}\fstop
				\end{equation}
				This is the key estimate to conclude the proofs of Theorems \ref{th:cutoff} and \ref{th:cutoff-window}.

				In case of Theorem \ref{th:cutoff}, the desired claim follows by choosing any $1\ll\psi\ll \frac{t_{\rm ent}}{t_{\rm rel}}\lesssim\log(n)$  and $t=(1-\varepsilon)t_{\rm ent}$ for any $\varepsilon\in(0,1)$. With this choice
			$s=16t_{\rm rel}\psi \ll t$ and the right-hand side of \eqref{eq:ub-cutoff} goes to zero thanks to Lemma \ref{lemma:main-WLLN} and $\psi\gg1$.
		
			Similarly, for Theorem \ref{th:cutoff-window}, we fix $\beta\in \R$, and set $t=t_{\rm ent}+\beta t_{\rm w}$.
			Recall that, since $\mu\ll \sqrt{\log n}$ by assumption \eqref{eq:prod-condition2}, there exists $\gamma\ll 1$ satisfying
			$\gamma \sqrt{\frac{\log n}{\mu}}\gg 1$. For such a sequence $\gamma\ll 1$, set $\psi:= \delta \sqrt{\mu\log n}$, with $\delta =\frac{\gamma}\mu>0$. With this definition, we have (cf.\ \eqref{eq:tw} and \eqref{eq:t-rel})
			\begin{equation}\label{eq:psi-gamma}
				1\ll  \psi \ll \frac{t_{\rm w}}{t_{\rm rel}}\comma
			\end{equation} 
and therefore $s = 16t_{\rm rel}\psi\ll t_{\rm w}$. In conclusion, thanks to Lemma \ref{lemma:main-window} and \eqref{eq:psi-gamma}, the right hand side of \eqref{eq:ub-cutoff} goes to zero by letting first $n\to\infty$ and then $\beta\to\infty$. This concludes the proof of both theorems.

\section{Cutoff profile}\label{sec:profile}
This section is  devoted to the proof of Theorem \ref{th:profile1}. 
	We start by noting that	the lower bound can be obtained arguing as in Section \ref{suse:lb-cutoff}, using Lemma \ref{lemma:main} in place of the rougher estimates in Lemmas \ref{lemma:main-WLLN} and \ref{lemma:main-window}. More precisely, 
	assume \eqref{eq:prod-condition2} and \eqref{eq:HP-noncutoff}, and, in the notation in Lemma \ref{lemma:main}, apply the estimate in \eqref{eq:general-lb} with $t=t_*(\beta):= t_{\rm ent}+\beta t_{\rm w}$, $a=e^\psi$, where $\psi:=\gamma \sqrt{\mu \log n}$, for some $\gamma > 0$.	Recall that  conditions \eqref{eq:prod-condition2}--\eqref{eq:HP-noncutoff} ensure that $t_{\rm w}\ll t_{\rm ent}$ and $\psi\gg 1$. Then, Lemma \ref{lemma:main} ensures that, for all $\beta\in \R$ and $\gamma >0$, 
	\begin{equation}
		\lim_{n\to\infty}\bigg|\Prob(|\zeta_t(U)|<\tfrac{e^\psi}n)- \big(1-\Xi_{\varrho}(\beta,\gamma)\big)\bigg|=0\comma
	\end{equation}
	and to conclude the proof of the lower bound it is enough to recall Remark \ref{rem:limit-Psi} and use the triangular inequality.

Therefore, in order to prove Theorem \ref{th:profile1} we are left  to show that, for all  $\beta\in \R$ and $\eps \in (0,1)$, we have
	\begin{equation}
		\Prob\big(d_{\rm TV}(t_*(\beta))>\Psi_\varrho(\beta)-\eps\big)\xrightarrow{n\to \infty}1\fstop
	\end{equation}
	For this purpose, we fix $\gamma > 0$ and apply \eqref{eq:lb} with $a=e^\psi>1$, $\psi:=\gamma \sqrt{\mu\log n}$,  and replacing therein $1-\eps$ by $ \Psi_\rho(\beta)-\eps/2$ (cf.\ \eqref{eq:Psi}). With this notation, recall that $w_t$, $t\ge 0$, stands for the sub-probability measure on $V$ obtained by removing from $\eta_t$ all the piles having mass smaller than $e^\psi/n$. We obtain
	\begin{align}
		&\Prob(d_{\rm TV}(t_*(\beta))>\Psi_{\varrho}(\beta)-\eps)\\
		&\ge \Prob(d_{\rm TV}(t_*(\beta))>\Psi_\varrho(\beta)-\eps\,,\, |w_{t_*(\beta)}|>\Psi_{\varrho}(\beta)-\eps/2)\\
		&\ge \Prob(\Psi_{\varrho}(\beta)-\eps/2-e^{-\psi}>\Psi_{\varrho}(\beta)-\eps\,,\, |w_{t_*(\beta)}|>\Psi_{\varrho}(\beta)-\eps/2)\\
		&= \Prob(|w_{t_*(\beta)}|>\Psi_{\varrho}(\beta)-\eps/2)\comma
	\end{align}
	where the last step holds for all $n\in \N$ sufficiently large, because $\psi\gg 1$. 
	
	To conclude the proof, in the next section we show that, for all $\beta\in\R$ and $\varepsilon\in (0,1)$,
	\begin{equation}\label{eq:last-2nd-mom}
		\Prob\(|w_{t_*(\beta)}|>\Psi_{\varrho}(\beta)-\varepsilon/2\)\to 1\,.
	\end{equation}
To do so, we use a second moment argument, showing that
\begin{equation}\label{eq:aim-2nd-mom}
	\E[|w_{t_*(\beta)}|^2]\sim \E[|w_{t_*(\beta)}|]^2\fstop
\end{equation}
Since $\E[|w_{t_*(\beta)}|]=\Prob\big(|\zeta_{t_*(\beta)}(U)|>\frac{e^\psi}{n}\big)$, by Lemma \ref{lemma:main}, we would obtain
	\begin{equation}\label{eq:recall-1st-mom}
		\E[|w_{t_*(\beta)}|]\sim \Psi_{\varrho}(\beta)\fstop
	\end{equation}
Then, \eqref{eq:last-2nd-mom} follows at once by \eqref{eq:aim-2nd-mom}, \eqref{eq:recall-1st-mom}, and Chebyshev inequality.
	\subsection{Proof of \eqref{eq:aim-2nd-mom}}
	Recall that we are under the assumptions of Theorem \ref{th:profile1}, and that $\psi:=\gamma \sqrt{\mu\log n}$, with  $\gamma >0$ fixed.
	Recall the definition of $t_*(\beta)$ in \eqref{eq:def-t-star-beta}, and note that under our assumptions we have
	\begin{equation}\label{eq:tw-trel}
		t_{\rm w}=  \left(1+\frac{\sigma}{\mu}\right)\frac{n\sqrt{\log n}}{\E[X]\sqrt \mu}\gtrsim t_{\rm rel}\sqrt{\frac{\log n}{\mu}}\gg t_{\rm rel}.
	\end{equation}
To ease the reading, in this subsection we use the shortcut $t=t_*(\beta)$.

The second moment in \eqref{eq:aim-2nd-mom} reads as 
	\begin{equation}\label{eq:rewrite1}
		\E[|w_t|^2]=\Prob\(|\zeta_t(U)|\ge \frac{e^\psi}{n}\,,|\zeta_t(U')|\ge \frac{e^\psi}{n} \)\,.
	\end{equation}
Introducing the events
	\begin{equation}\label{eq:SU-SU'}
		\cS_U:=\left\{|\zeta_t(U)|< \frac{e^\psi}{n} \right\}\comma\qquad \cS_{U'}:=\left\{|\zeta_t(U')|< \frac{e^\psi}{n} \right\}\,.
	\end{equation}
	We shall prove $\Prob(\cS_U^\complement \cap \cS_{U'}^\complement)\sim \Prob(\cS_U^\complement)^2$, or, equivalently, 
	\begin{equation}\label{eq:final}
		\Prob(\cS_U\cap \cS_{U'})\sim \Prob(\cS_U)^2\fstop
	\end{equation}
	
	Let us introduce another timescale, $t_{\rm sep}$, such that
	\begin{equation}\label{eq:def-tsep}
		t_{\rm rel}\ll t_{\rm sep}\ll t_{\rm rel}\sqrt{\frac{\log(n)}{\mu}}\comma
	\end{equation}
and call 
\begin{equation}\tau_{\rm sep}:=\inf\{s\ge 0\mid e_s(U)\neq e_s(U') \}\fstop 
	\end{equation}
	Consider the event
	\begin{equation}
		\cE_1:=\{\tau_{\rm sep}\le t_{\rm sep} \}\fstop
	\end{equation}
Since the number of effective updates for a chunk up to time $t_{\rm sep}$ is distributed  as $T\sim {\rm Bin}(t_{\rm sep},\Prob(Z>1))$,  and the probability that two chunks sit together after an update is at most $1/2$, we get
	\begin{equation}\label{eq:usual-l2}
		\begin{split}
		\Prob(\cE_1^\complement)
		&\le\sum_{r=0}^{t_{\rm sep}}2^{-r}\,\Prob({\rm Bin}(t_{\rm sep},\Prob(Z>1))=r)\ll 1\comma
		\end{split}
	\end{equation}
where the last step used $\Prob(Z>1)=\frac{\E[X]}{n} \asymp t_{\rm rel}^{-1}$, as well as the first inequality in \eqref{eq:def-tsep}.

	We now further introduce two diverging scales, $\theta$ and $K$, satisfying
	\begin{equation}\label{eq:def-theta-K}
		\theta\gg\frac{t_{\rm sep}}{t_{\rm rel}}\comma\qquad \log K-\mu\gg (\sigma+1)\sqrt{\theta}\fstop
	\end{equation}
	Define
	\begin{equation}
		\cE_2:=\{\text{the number of updates involving either $U$ or $U'$ before $t_{\rm sep}$ is at most $\theta$} \}\fstop
	\end{equation}
	Then, by a union bound and arguing as in \eqref{eq:usual-l2}, we obtain
	\begin{equation}\label{eq:E2}
		\Prob(\cE_2^\complement)\le 2\Prob({\rm Bin}(t_{\rm sep},\Prob(Z>1))\ge \theta)\ll 1\comma
	\end{equation}
where the last step follows from the first inequality in \eqref{eq:def-theta-K}.
	Moreover, consider
	\begin{equation}
		\cE_3:=\{\text{all updates involving either $U$ or $U'$ within $t_{\rm sep}$ are made by blocks of size $\le K$} \}\fstop
	\end{equation}
	Then,	by a union bound and Markov inequality, we get
	\begin{equation}
		\Prob(\cE_3\cap\cE_2)\ge 1-2\Prob\({\rm Bin}(\theta,\Prob(Y>K)) >0\)\ge 1-2\theta \Prob(Y>K)\comma
	\end{equation}
	Since Chebyshev inequality and the second inequality in \eqref{eq:def-theta-K} ensure
	\begin{equation}
		\Prob(Y>K)=\Prob(\log Y>\log K)= \Prob\(\log Y-\mu>\log K-\mu\)\le\frac{\sigma^2}{(\log K-\mu)^2}\ll \frac1\theta\comma
	\end{equation}
	we obtain
	\begin{equation}\label{eq:E1-2-3}
	\Prob(\cE_2\cap \cE_3)\sim 1\fstop
	\end{equation}
Hence, $\Prob(\cE_1\cap \cE_2\cap \cE_3)\sim 1$, and observe that, under $\cE_1\cap\cE_2\cap \cE_3$, we have (deterministically)
	\begin{equation}\label{eq:large-pile}
		|\zeta_{t_{\rm sep}}(U)|,|\zeta_{t_{\rm sep}}(U')|\ge K^{-\theta}
		\fstop
	\end{equation}
	Therefore, we proved that with high probability $U$ and $U'$ ever visited different sites (hence, in particular, sit on distinct piles) within time $t_{\rm sep}$ and their piles at that time are still relatively large (if $\theta$ and $K$ are not too large). More precisely, by \eqref{eq:prod-condition2} and \eqref{eq:HP-noncutoff}, we can choose $\theta$ and $K$  so to satisfy \eqref{eq:def-theta-K}, as well as
	\begin{equation}\label{eq:cond-new}
		\theta\log(K)\ll\psi \asymp \sqrt{\mu\log n}\fstop
	\end{equation}

	As we are interested in the probability of $\cS_U\cap \cS_{U'}$, we have, by \eqref{eq:usual-l2} and \eqref{eq:E1-2-3},
	\begin{equation}
			\Prob\big(\cS_U\cap \cS_{U'}\big) \sim \Prob\big(\cS_U\cap \cS_{U'}\mid \cE_1\cap \cE_2\cap \cE_3\big)\comma
	\end{equation} 
	which, by \eqref{eq:large-pile} and the Markov property, is further estimated as follows
	\begin{align}\label{eq:int-split}
		\begin{aligned}
	 \Prob(\cS_U\cap \cS_{U'})&\ge	\inf_{z,z',r} \widehat\Prob_{z,z'}\bigg(|\zeta_{t-r}(U)|<\frac{e^{\psi}}n\, ,\,|\zeta_{t-r}(U')|<\frac{e^{\psi}}n \bigg)  \comma\\ 
	\Prob(\cS_U\cap \cS_{U'}) &\le \sup_{z,z',r}\widehat\Prob_{r,r'}\bigg(|\zeta_{t-r}(U)|\le\frac{e^{\psi}}n\, ,\,|\zeta_{t-r}(U')|<\frac{e^{\psi}}n \bigg) \comma
		 \end{aligned}
	\end{align}
	where both $\inf$ and $\sup$ above run over $(z,z',r)\in [K^{-\theta},1]^2\times \{0,\ldots, t_{\rm sep}\}$, whereas $\widehat \Prob_{z,z'}$ denotes the law of two chunks, initially placed on two different piles of size $z$ and $z'$, respectively.
	
	We now prove the following claim: letting $(A_s)_{s\ge 1}\subset V$ denote the random sequence of blocks, and
	\begin{equation}
		\cE_4:=\{{\nexists} s\in(\tau_{\rm sep},t]\ \text{s.t.}\ e_s(U)\,,\,e_s(U')\in A_s\}\comma
	\end{equation}
	we have
	\begin{equation}\label{eq:E4}
		\Prob(\cE_4)\to 1\fstop
	\end{equation}
	By Markov inequality and the definition $t=t_*(\beta)\sim t_{\rm ent}$ (cf.\ \eqref{eq:t-ent}), we obtain
\begin{equation}
	\Prob(\cE_4^\complement )\le t \sum_{k=2}^np_X(k)\frac{\binom{n-2}{k-2}}{\binom{n}{k}}\lesssim \frac{t_{\rm ent}\,\E[X^2]}{n^2} \fstop
\end{equation}
	Hence, it suffices to prove that the right-hand side above vanishes under our assumptions. This is the content of the following lemma.
	\begin{lemma}\label{lemma:6}
		Under conditions \eqref{eq:prod-condition2} and \eqref{eq:HP-noncutoff}, we have
	\begin{equation}
	 {t_{\rm ent}\,\E[X^2]}\ll n^2\fstop
	\end{equation}
	\end{lemma}
	\begin{proof}
		From the definition of $t_{\rm ent}$ in \eqref{eq:t-ent}, it suffices to prove
			\begin{equation}	
		\label{eq:to-be-checked}
		\frac{\E[X^2]}{\E[X \log X]} \ll \frac{n}{\log n}\fstop		
		\end{equation}
		For any $A>0$, and since $y\in [2,\infty)\mapsto y/\log(y)$ is increasing, we can factorize the numerator in \eqref{eq:to-be-checked} as follows
		\begin{align}
			\E[X^2]&=\E[X^2 \IND(X/\log(X)\le A)]+ \E[X^2 \IND(X/\log(X)> A)]\\
			&\le A\, \E[X\log(X)]+\frac{n}{\log(n)}\,\E[X\log(X)\IND(X/\log(X)>A)]\comma
		\end{align}
		from which it follows that
		\begin{equation}\label{eq:middle}
			\frac{\E[X^2]}{\E[X\log(X)]}\le A + \frac{n}{\log(n)}\frac{\E[X\log(X)\IND(X/\log(X)>A)]}{\E[X\log(X)]}\fstop
		\end{equation}
	Recalling that $X\ge 2$,  we further bound
		\begin{equation}\label{eq:next-bound}
			\frac{\E[X\log(X)\IND(X/\log(X)>A)]}{\E[X\log(X)]}\le \frac{\E[X\log(X)\IND(X>A)]}{\E[X\log(X)]}\fstop
		\end{equation}
		To prove \eqref{eq:to-be-checked}, observe that
		\begin{equation}
			\frac{\E[X\log^2(X)]}{\E[X\log(X)]}\ge 	\frac{\E[X\log^2(X)\IND(X>A)]}{\E[X\log(X)]}\ge \log(A)\frac{\E[X\log(X)\IND(X>A)]}{\E[X\log(X)]}\comma
		\end{equation}
		and now  argue by contradiction choosing, e.g., $A=\sqrt n$. Assume that the ratio on the right-hand side of \eqref{eq:next-bound} does not go to zero when $n\to\infty$. Then, we would have
		\begin{equation}
			{\E[X\log^2(X)]}\gtrsim	 {\log(A)}\E[X\log(X)]\quad\iff\quad \frac{\E[X\log^2(X)]}{\E[X]}\gtrsim	 \mu \log(A)\fstop
		\end{equation}
		Since $A=\sqrt{n}$, by exploiting \eqref{eq:prod-condition2} and subtracting $\mu^2\ll \mu \log (n)$ on both sides, we then obtain
		\begin{equation}
			\sigma^2\gtrsim \mu\log(n) \comma
		\end{equation}
		which contradicts \eqref{eq:HP-noncutoff}. We can therefore deduce that it must be the case that
		\begin{equation}\label{eq:conclusion}
			\frac{\E[X\log(X)\IND(X>A)]}{\E[X\log(X)]}\to 0\comma
		\end{equation}
		and the proof of \eqref{eq:to-be-checked} follows from \eqref{eq:middle} and \eqref{eq:conclusion} with $A=\sqrt n$.
	\end{proof}

In what follows, fix $z,z'\in [K^{-\theta},1]^2$ and $t'=t-r\in \{t-t_{\rm sep},\ldots, t\}$. Going back to \eqref{eq:int-split}, and combining it with \eqref{eq:E4}, we have
\begin{align}
	&\widehat \Prob_{z,z'}\bigg(|\zeta_{t'}(U)|<\frac{e^{\psi}}n\, ,\,|\zeta_{t'}(U')|<\frac{e^{\psi}}n \bigg)\\
	 &\qquad= \widehat \Prob_{z,z'}\bigg(|\zeta_{t'}(U)|<\frac{e^{\psi}}n\, ,\,|\zeta_{t'}(U')|<\frac{e^{\psi}}n\,,\, \cE_4 \bigg)+o(1)\\
	&\qquad=\Prob\bigg(\prod_{s=1}^{t'} Z_s^{-1}<\frac{e^\psi}{zn}\,,\,  \prod_{s=1}^{t'}(Z_s')^{-1}<\frac{e^\psi}{z'n}\bigg)+o(1)\comma\label{eq:int-split2}
\end{align}
where the random vectors $((Z_s,Z_s'))_{s=1}^{t'}$ are i.i.d.,  marginally distributed as in \eqref{eq:Z-law}, and jointly as
\begin{align}
	&\Prob(Z_s=Z_s'=1\mid X_s) = \frac{(n-X_s)(n-X_s-1)}{n(n-1)}\sim1-\frac{2X_s}{n} +\frac{X^2_s-X_s}{n^2}\comma\\
	& \Prob(Z_s=X_s\,,\, Z_s'=1\mid X_s)= \Prob(Z_s=1\,,\, Z_s'=X_s\mid X_s) \\
	&\qquad =\frac{n(n-1)-(n-X_s)(n-X_s-1)}{2n(n-1)}\sim\frac{X_s}{n}-\frac{X^2_s-X_s}{2n^2} \comma\\
	& \Prob(Z_s=Z_s'=X_s\mid X_s)=\frac{X_s(X_s-1)}{n(n-1)} \sim\frac{X_s(X_s-1)}{n^2}\fstop
\end{align} 
We shall prove
In particular, given $X_s$,  $Z_s$ and $Z_s'$ are \textit{not} independent. We shall compare this sequence with $((Z_s,\widetilde Z_s))_{s=1}^{t'}$, with $Z_s$ and $\widetilde Z_s$  i.i.d., conditionally on $X_s$. More precisely,
	\begin{align}
		&\Prob(Z_s=\widetilde Z_s=1\mid X_s) = \(1-\frac{X_s}{n}\)^2= 1-\frac{2X_s}{n}+\frac{X_s^2}{n^2}
		\comma\\
		& \Prob(Z_s=X_s\,,\, \widetilde Z_s=1\mid X_s)= \Prob(Z_s=1\,,\, Z_s'=X_s\mid X_s) \\
		&\qquad =\(1-\frac{X_s}{n}\)\frac{X_s}{n}=\frac{X_s}{n}-\frac{X^2_s}{n^2} \comma\\
		& \Prob(Z_s=\widetilde Z_s=X_s\mid X_s)=\frac{X_s^2}{n^2}\fstop
	\end{align} 
As a consequence, given $(X_s)_{s\le t'}$,  $(Z_s,Z_s')$ can be coupled with $(Z_s,\widetilde Z_s)$ at a ${\rm TV}$-cost
\begin{equation}\label{eq:coupling}
\big\|(Z_s,Z_s'\mid X_s)_{s\le t'}- (Z_s,\widetilde Z_s\mid X_s)_{s\le t'}\big\|_{\rm TV}	\le 2\sum_{s=1}^{t'} \frac{X_s^2}{n^2}\le 2 \sum_{s=1}^{2t_{\rm ent}}\frac{X_s^2}{n^2}\comma
\end{equation}
whose expectation vanishes by Lemma \ref{lemma:6}. Henceforth, by Markov inequality, we have
\begin{equation}\label{eq:int-split-2}
\Prob\bigg(\prod_{s=1}^{t'} Z_s^{-1}<\frac{e^\psi}{zn}\,,\,  \prod_{s=1}^{t'}(Z_s')^{-1}<\frac{e^\psi}{z'n}\bigg)= \Prob\bigg(\prod_{s=1}^{t'} Z_s^{-1}<\frac{e^\psi}{zn}\,,\,  \prod_{s=1}^{t'}\widetilde Z_s^{-1}<\frac{e^\psi}{z'n}\bigg) + o(1)\fstop
\end{equation}
By the independence of 
$(Z_s)_{s\ge 1}$ and $(\widetilde Z_s)_{s\ge 1}$ conditioned to $(X_s)_{s\ge 1}$, we have
\begin{align}
&	\Prob\bigg(\prod_{s=1}^{t'} Z_s^{-1}<\frac{e^\psi}{zn}\,,\,  \prod_{s=1}^{t'}\widetilde Z_s^{-1}<\frac{e^\psi}{z'n}\bigg) \\
&= \E\bigg[\Prob\bigg(\prod_{s=1}^{t'} Z_s^{-1}<\frac{e^\psi}{zn}\bigg|(X_s)_{s=1}^{t'}\bigg)\Prob\bigg(\prod_{s=1}^{t'} \widetilde Z_s^{-1}<\frac{e^\psi}{z'n}\bigg|(X_s)_{s=1}^{t'}\bigg)\bigg]\fstop
\end{align}
We want to show, uniformly over $z,z'\in [K^{-\theta},1]$ and $t'\in \{t-t_{\rm sep},\ldots,t\}$, that
\begin{align}\label{eq:first-second-approximation}
	\begin{aligned}
	&\E\bigg[\Prob\bigg(\prod_{s=1}^{t'} Z_s^{-1}<\frac{e^\psi}{zn}\bigg|(X_s)_{s=1}^{t'}\bigg)\Prob\bigg(\prod_{s=1}^{t'} \widetilde Z_s^{-1}<\frac{e^\psi}{z'n}\bigg|(X_s)_{s=1}^{t'}\bigg)\bigg]\\
	&\qquad  \sim \E\bigg[\Prob\bigg(\prod_{s=1}^{t'} Z_s^{-1}<\frac{e^\psi}{zn}\bigg|(X_s)_{s=1}^{t'}\bigg)\bigg]\E\bigg[\Prob\bigg(\prod_{s=1}^{t'} \widetilde Z_s^{-1}<\frac{e^\psi}{z'n}\bigg|(X_s)_{s=1}^{t'}\bigg)\bigg]\\
	&\qquad\sim \Prob(\cS_U)^2 \fstop
	\end{aligned}
\end{align}
We now separately prove  both approximations in \eqref{eq:first-second-approximation}.

\smallskip
\noindent \textit{First approximation in \eqref{eq:first-second-approximation}.}
In what follows, write $\bar X:=(X_1,\ldots, X_{t'})$. In order to prove the first approximation in \eqref{eq:first-second-approximation}, by Cauchy-Schwarz inequality, it suffices to show  
\begin{equation}\label{eq:var-vanishes}
\sup_{z,t'} {\rm Var}\big(f(\bar X)\big)\ll 1\comma\quad \text{with}\ f(\bar X)=f_{z,t'}(\bar X):= \Prob\bigg(\prod_{s=1}^{t'} Z_s^{-1}<\frac{e^\psi}{zn}\bigg|\bar X\bigg)\comma
\end{equation} 
where the supremum runs over $	z \in [K^{-\theta},1]$ and $t'\in \{t-t_{\rm sep},\ldots, t\}$.
Since $\bar X$ has i.i.d.\ components, we may estimate the above variance by Efron--Stein inequality (see, e.g., \cite[Theorem 3.1]{boucheron_lugosi_massart_concentration_2013}):
\begin{align}\label{eq:efron-stein}
	{\rm Var}(f(\bar X))&\le \frac{1}{2}\sum_{s=1}^{t'}\E\big[\big(f(\bar X)-f(\bar X^{(s)})\big)^2\big]= \frac{t'}{2}\,\E\big[\big(f(\bar X)-f(\bar X^{(1)})\big)^2\big]
	\comma
\end{align}
where, given $\bar X$, we defined $\bar X^{(s)}:=(X_1,\ldots, X_{s-1},X_s',X_{s+1},\ldots, X_{t'})$, with $X_s'$ being an independent copy of $X_s$, whereas the second step follows by the symmetry of the function $f$. Let us introduce some shorthand notation: $W:= \sum_{s=2}^{t'}\log(Z_s)$, $W_1:= \log(Z_1)$,  $W_1':=\log(Z_1')$, and $h:= \log n-\psi+\log z$.  By the tower property, we obtain
\begin{align}
	&f(\bar X)-f(\bar X^{(1)})\\
	&= \Prob(W_1+W>h\mid \bar X) - \Prob(W_1'+W>h\mid \bar X^{(1)})\\
	&= \E[\Prob(W_1+W>h\mid X_1,W)-\Prob(W_1'+W>h\mid X_1',W)\mid X_2,\ldots, X_{t'}]\\
	&= \sum_{w=0}^\infty \Prob(W=w\mid X_2,\ldots, X_{t'}) \left\{\Prob(W_1+w>h\mid X_1)-\Prob(W_1'+w>h\mid X_1') \right\}\fstop
\end{align} 
By Jensen inequality, we have
\begin{align}
	&\big(f(\bar X)-f(\bar X^{(1)})\big)^2\\
	&\le \sum_{w=0}^\infty \Prob(W=w\mid X_2,\ldots, X_{t'}) \left(\Prob(W_1+w>h\mid X_1)-\Prob(W_1'+w>h\mid X_1')\right)^2\comma
\end{align}
and, further taking expectation with respect to $\bar X$ and $\bar X^{(1)}$, we get, by the independence of the $X$-variables, 
\begin{align}
	&\E\big[\big(f(\bar X)-f(\bar X^{(1)})\big)^2\big]\\
	&\le \sum_{w=0}^\infty \E\big[\Prob(W=w\mid X_2,\ldots, X_{t'})\big] \E\big[\big(\Prob(W_1>h-w\mid X_1)-\Prob(W_1'>h-w\mid X_1')\big)^2\big]\\
	&\le 2 \sum_{w=0}^h \E\big[\Prob(W=w\mid X_2,\ldots, X_{t'})\big]\E\big[\big(\Prob(W_1>h-w\mid X_1)\big)^2\big]\fstop
	\label{eq:efron-stein-final}
\end{align}
Since, by \eqref{eq:Z-law} (recall $W_1\in \{0,\log(X_1)\}$) and $w\le h$, we have
\begin{equation}
\Prob(W_1>h-w\mid X_1)= \frac{X_1}{n}\, \IND\{\log(X_1)>h-w\}\le \frac{X_1}{n}\comma
\end{equation}
the right-hand side of \eqref{eq:efron-stein-final} is further bounded above by
\begin{align}
	2\sum_{w=0}^h \E\big[\Prob(W=w\mid X_2,\ldots, X_{t'})\big]\frac{\E\big[X^2\big]}{n^2}\le 2\,\frac{\E\big[X^2\big]}{n^2}\fstop
\end{align}
In conclusion, plugging this back into \eqref{eq:efron-stein}, we obtain, uniformly over $z\in [K^{-\theta},1]$ and $t'\in \{t-t_{\rm sep},\ldots, t\}$,
\begin{equation}
	{\rm Var}(f(\bar X))\le t'\, \frac{\E[X^2]}{n^2}\lesssim t_{\rm ent} \frac{\E[X^2]}{n^2}\comma
\end{equation}
and, thus, the desired claim in \eqref{eq:var-vanishes} by Lemma \ref{lemma:6}.

\smallskip
\noindent \textit{Second approximation in \eqref{eq:first-second-approximation}.} Recalling the definition of the event $\cS_U$ from \eqref{eq:SU-SU'}, and the identity in Lemma \ref{lemma:size-pile} (see also \eqref{eq:size-typical-Z}), we need to show, uniformly over $z\in [K^{-\theta},1]$ and $t'\in \{t-t_{\rm sep},\ldots, t\}$, 
\begin{equation}
	\Prob\bigg(|\zeta_t(U)|<\frac{e^\psi}{n}\bigg) = \Prob\bigg(\prod_{s=1}^t Z_s^{-1}<\frac{e^\psi}{n}\bigg)\sim \Prob\bigg(\prod_{s=1}^{t'} Z_s^{-1}<\frac{e^\psi}{zn}\bigg)=\Prob\bigg(|\zeta_{t'}(U)|<\frac{e^{\psi}}{zn}\bigg)\fstop
\end{equation}
Thanks to \eqref{eq:cond-new}, we have $ |\log z|\ll \psi\ll \log n$, ensuring
\begin{equation}
	\Prob\bigg(\prod_{s=1}^{t'} Z_s^{-1}<\frac{e^\psi}{zn}\bigg)\sim \Prob\bigg(\prod_{s=1}^{t'} Z_s^{-1}<\frac{e^\psi}{n}\bigg)\fstop
\end{equation} 
Similarly, since $t_{\rm sep}/ t_{\rm rel}\gg 1$ diverges arbitrarily slow (satisfying, in particular, \eqref{eq:def-tsep}), for any  $m\gg 1$ satisfying $\mu \frac{t_{\rm sep}}{t_{\rm rel}}\ll\log m\ll \psi\asymp \sqrt{\mu\log n}$, Markov inequality yields	
\begin{equation}
	\Prob\bigg(\sum_{s={t'}+1}^t \log Z_s\ge  \log m\bigg) \le  \frac{t_{\rm sep}\,\E[\log Z]}{\log m}  = \frac{t_{\rm sep}\,\E[X\log X]}{n\log m}\ll 1\fstop
\end{equation}
Since $\log m\ll \psi$, we finally get
\begin{equation}
	\Prob\bigg(\prod_{s=1}^{t'} Z_s^{-1}<\frac{e^\psi}{n}\bigg)\sim \Prob\bigg(\prod_{s=1}^{t} Z_s^{-1}<\frac{e^\psi}{n}\bigg)\fstop
\end{equation}
and this concludes the proof of 
\eqref{eq:final} and, thus, of \eqref{eq:aim-2nd-mom}.
\qed

\section{Deterministic block size}\label{sec:deterministic}
In this section, we assume that the blocks' sizes are deterministically equal to $k$. Since the case $\frac{\log k}{\log n}\to 0$
is already covered by Theorem \ref{th:profile1}, we focus on the alternative scenario, namely, $\limsup_{n\to \infty} \frac{\log k}{\log n}> 0$. For simplicity of exposition,  we shall pass to a converging subsequence, and  assume 
\begin{equation}\label{eq:eps-bar}
	\exists\lim_{n\to \infty}\frac{\log k}{\log n}>0\fstop
	\end{equation}
By  Propositions \ref{pr:prod-cond} and \ref{pr:tent-tmix}, cutoff does not occur in this case, and $d_{\rm TV}$ decays non-trivially on the timescale $n/k$, which coincides, at first order, with $t_{\rm rel}=\frac{n-1}{k-1}$. 

A crucial role is played by the first time in which the site initially possessing all the mass is involved in a block update.
Letting $\tau_{\rm start}$ denote this time, we have
\begin{equation}
	\tau_{\rm start} \overset{\rm d}={\rm Geo}(k/n)\comma
\end{equation} and, thus,  
\begin{equation}
	\Prob(\tau_{\rm start}>\lfloor \beta\tfrac{n}{k}\rfloor)=\left(1-\frac{k}n\right)^{\lfloor \beta\frac{n}k\rfloor}\comma\qquad \beta > 0\fstop
\end{equation}
As a consequence, if $\frac{k}n \to 0$,  $\frac{k}{n}\tau_{\rm start}\overset{{\rm d}}{\longrightarrow}{\rm Exp}(1)$; otherwise, if  $\frac{k}n \to c\in(0,1]$,  $\tau_{\rm start}\overset{{\rm d}}{\longrightarrow}{\rm Geo}(c)$.
 The core of the mixing phenomenology  begins after this first update. In what follows, for simplicity, we will often denote by $\eta^{\rm start}$ a configuration having $k$ vertices of mass $k^{-1}$ (the labels of these vertices is not relevant to our scope).

We divide the analysis of the scenario in \eqref{eq:eps-bar} into two regimes, dealt with in Sections \ref{suse:linear-size} and \ref{suse:poly-size}, respectively:
\begin{equation}\label{eq:5/6}
	\lim_{n\to \infty}\frac{\log k}{\log n} \in \left[5/6,1\right]\comma\qquad
\lim_{n\to \infty} \frac{\log k}{\log n}\in \left(0,5/6\right]\fstop
\end{equation}

	\begin{remark}
	The value $5/6$ in \eqref{eq:5/6} does not have any particular meaning, and the proof actually works for any constant strictly larger than $1/2$. In other words, the approach employed for the first regime can be tuned to work for any value of $k> n^{1/2+o(1)}$. Nevertheless, we found more convenient to present the proof up to this arbitrary threshold, while incorporating the values of $n^{1/2}\ll k\ll n^{5/6}$ into the second regime.
\end{remark}
Before starting our analysis, let us recall the following elementary identity
\begin{equation}
\Pr\(\tau_j> \beta \) = 	\Pr\({\rm Poi}(\beta)\le j-1  \)\comma\qquad j\in \N\comma \beta > 0\comma
\end{equation}
where ${\rm Poi}(\beta)$ stands for a Poisson random variable of parameter $\beta$.

\subsection{First regime in \eqref{eq:5/6}}\label{suse:linear-size}

	In this section, we  consider $k=c n$, for some sequence $c=c_n\in(0,1)$ satisfying
	\begin{equation}\label{HP:linear}
		\lim_{n\to \infty}\frac{\log k}{\log n} = \lim_{n\to \infty} 1-\frac{\log \frac1c}{\log n}\ge 5/6\fstop
	\end{equation}
	Note that we skip the trivial case $c\sim 1$, for which, w.h.p., we have $\tau_{\rm start}=1$  and $d_{\rm TV}(\tau_{\rm start})=1-c\to 0$.	

\begin{proposition} Letting $k=cn$, under the assumption in \eqref{HP:linear}, we have, for all $\beta\ge  0$, 
	\begin{equation}
		d_{\rm TV}(\tau_{\rm start}+\lfloor \beta\tfrac{n}k\rfloor-1)  \overset{\Prob}\longrightarrow \left(1-\bar c\right)^\beta\comma\quad \text{if}\  \lim_{n\to \infty}c\eqqcolon \bar c>0\comma
	\end{equation}
	whereas
	\begin{equation}
		d_{\rm TV}(\tau_{\rm start}+\lfloor \beta \tfrac{n}{k} \rfloor-1) 
		\overset{\Prob}\longrightarrow e^{-\beta}\comma\quad \text{if}\ \lim_{n\to \infty}c=0\fstop
	\end{equation}
\end{proposition}
\begin{proof}		
	Right after time $\tau_{\rm start}$, the vertex set $V$ can be partitioned into two subsets: the vertices having mass $k^{-1}$ and those having mass zero.  Call $B_0^{\rm big}$ the former subset, $B_0^{\rm void}$ the latter one. Clearly, $|B_0^{\rm big}|=k$, and therefore
	\begin{equation}
		d_{\rm TV}(\tau_{\rm start})=1-c\fstop
	\end{equation}
	Letting $A_1$ be the subset of vertices sampled at time $\tau_{\rm start}+1$ to perform an update, we further  
	define $A_1^\star\subset B_0^{\rm big}$ as the subset of vertices in $A_1$ having mass $k^{-1}$. Then,
	\begin{equation}
		|A_1^\star|\overset{\rm d}{=}{\rm HyperGeo}(n,k,k)\fstop
	\end{equation}
	Suppose that $|A_1^\star|=j$. Then, at time $\tau_{\rm start}+1$, there are $k-j$ vertices with mass $k^{-1}$, $k$ vertices with mass $jk^{-2}$, while the remaining vertices have mass zero. 
	By concentration of hypergeometric random variables we have (see, e.g., \cite{serfling_probability_1974}), for all $n$ large enough,
	\begin{equation}
		\Prob(||A_1^\star|- c|B^{\rm big}_{0}| |> k^{2/3} )\le \exp(-k^{1/4} )\fstop
	\end{equation}
	In other words, there exists a random variable $E_1\in\R$ satisfying
	\begin{equation}
		|A_1^\star|= \frac{k}n|B_0^{\rm big}|\left(1+E_1\right)
	=c^2 n\left(1+  E_1\right)\comma
	\end{equation}
	and, by \eqref{HP:linear},
	\begin{equation}\label{eq:def-p}
		\Prob(|E_1|> n^{-1/5} )\le p\coloneqq\exp(-k^{1/4} )\fstop
	\end{equation}
	Therefore, after the $(\tau_{\rm start}+1)$-th update, the vertices in $V$ can be partitioned into three sets: 
	\begin{align}\label{eq:B1}
		\begin{aligned}
		B_1^{\rm big}&\coloneqq \{x\in V\mid \eta_{\tau_{\rm start}+1}(x)=k^{-1} \}\comma\\
		B_1^{\rm void}&\coloneqq \{x\in V\mid \eta_{\tau_{\rm start}+1}(x)=0 \}\comma\\
		B_1^{\rm eq}&\coloneqq \{x\in V\mid \(1-|E_1|\)n^{-1}\le\eta_{\tau_{\rm start}+1}(x)\le \(1+|E_1|\)n^{-1} \}\comma
		\end{aligned}
	\end{align}
	having sizes
	\begin{align}
		|B_1^{\rm big}|&=k-|A_1^\star|\in \big[\(1-|E_1|\)c(1-c)n\,,\,\(1+|E_1|\)c(1-c)n \big]\comma\\
		|B_1^{\rm eq}|&=k\fstop
	\end{align}
	Henceforth,
	\begin{equation}
		|B_1^{\rm big}|(k^{-1}-n^{-1})\le	d_{\rm TV}(\tau_{\rm start}+1)\le |B_1^{\rm big}|(k^{-1}-n^{-1})+|B_1^{\rm eq}| \frac{|E_1|}{n}\comma
	\end{equation}
	and, in conclusion, with probability at least $p$ (defined as in \eqref{eq:def-p}), 
	\begin{equation}
		d_{\rm TV}(\tau_{\rm start}+1)=(1+O(n^{-1/5}))(1-c)^2+O\big( n^{-1/5}\big)\fstop
	\end{equation}

		We can now argue by induction. Let us fix some $t\ge 2$, and assume that at time $t-1$, with probability at least $(1-p)^{t-1}$, we have
	\begin{align}
		|B_{t-1}^{\rm big}|&=(1+O(n^{-1/5}))^{t-1}c(1-c)^{t-1}n\comma\\
		|B_{t-1}^{\rm eq}|&=(1+O(n^{-1/5}))^{t-1}\(1-(1-c)^{t-1}\)n\comma
	\end{align}
	where we used a notation analogous to that in \eqref{eq:B1}, that is,
	\begin{align}
		B_{t-1}^{\rm big}&\coloneqq \{x\in V\mid \eta_{\tau_{\rm start}+t-1}(x)=k^{-1}\}\\
		B_{t-1}^{\rm void}&\coloneqq \{x\in V\mid \eta_{\tau_{\rm start}+t-1}(x)=0\}\comma
		\\
		B_{t-1}^{\rm eq}&\coloneqq 
		V\setminus\{B_{t-1}^{\rm big}\cup B_{t-1}^{\rm void}\}
		\fstop
	\end{align}
	Observe that this implies that, with probability at least $(1-p)^{t-1}$, 
	\begin{equation}
		d_{\rm TV}(\tau_{\rm start} +t-1) = (1+O(n^{-1/5}))\,(1-c)^{t} +O(n^{-1/5})\fstop
	\end{equation}
	
	Under this event at time $t-1$,  when we update vertices in the random block $A_t$, we get, with probability at least $1-p$  again by concentration of hypergeometrics, 
	\begin{align}
		|A_t^\star|&=\frac kn\,(1+O(n^{-1/5}))\,|B_{t-1}^{\rm big}|=(1+O(n^{-1/5}))^t\, c^2(1-c)^{t-1}\,n\comma\\
		|A_t^{\rm eq}|&=\frac kn\,(1+O(n^{-1/5}))\,|B_{t-1}^{\rm eq}| = (1+O(n^{-1/5}))^t\,c (1-(1-c)^{t-1})\,n\fstop
	\end{align}
	Altogether, with probability at least $(1-p)^t$, each vertex in $A_t$ has mass equal to
	\begin{equation}(1+O(n^{-1/5}))^t\left[c^2(1-c)^{t-1}n k^{-1} + c (1-(1-c)^{t-1})\right]k^{-1}=(1+O(n^{-1/5}))^t\, n^{-1}\comma
		\end{equation}
	and, furthermore, 
	\begin{align}
		|B_{t}^{\rm big}|&=(1+O(n^{-1/5}))^t\,[|B_{t-1}^{\rm big}|-|A_t^\star|]=(1+O(n^{-1/5}))^t\,c(1-c)^{t}n\comma\\
		|B_{t}^{\rm eq}|&=(1+O(n^{-1/5}))^t(|B_{t-1}^{\rm eq}|-|A_t^{\rm eq}|+k)=(1+O(n^{-1/5}))^t\,(1-(1-c)^{t})\,n\fstop
	\end{align}
	Consequently, with probability at least $(1-p)^t$, we have
	\begin{equation}\label{eq:tv-linear}
		d_{\rm TV}(\tau_{\rm start}+t)=(1+O(n^{-1/5}))^t\,(1-c)^{t+1}+O\big(n^{-1/5} \big)\fstop
	\end{equation}
	
	In conclusion, taking $\bar t(\beta)=\lfloor\beta \frac{n}{k}\rfloor= \lfloor\beta c^{-1}\rfloor$ for any $\beta > 0$, if $c \not\to 0$, \eqref{eq:tv-linear} yields
	\begin{equation}
		d_{\rm TV}(\tau_{\rm start}+\bar t(\beta)) \sim (1-c)^{\bar t(\beta)+1}\comma
	\end{equation}
	and the proof is complete in this case.
	If $c \to 0$ (yet, satisfying \eqref{HP:linear}, thus, $c\gg n^{-1/5}$), we rather obtain 
	\begin{align*}
		d_{\rm TV}(\tau_{\rm start}+\bar{t}(\beta))\sim e^{-\beta} \fstop 
	\end{align*}
	This concludes the proof of the proposition. 
	\end{proof}

		\subsection{Second regime in \eqref{eq:5/6}}\label{suse:poly-size}
		In this section, we consider $k=n^\delta$, for some sequence $\delta = \delta_n\in (0,1)$ such that the second condition in \eqref{eq:5/6}  holds true, that is, 
		\begin{equation}\label{HP:polynomial}
			\delta\longrightarrow \bar \delta\in (0,5/6]\fstop
		\end{equation}
		
	\begin{proposition}\label{pr:poly}Letting $k=n^\delta$, under the assumption in \eqref{HP:polynomial}, we have, for all $\beta \ge 0$ and w.h.p., 
			\begin{equation}
			\Pr\left({\rm Poi}(\beta)<\frac{1-\delta}{\delta}\right)	\le d_{\rm TV}(\tau_{\rm start}+\lfloor \beta \tfrac{n}k\rfloor)\le \Pr\left({\rm Poi}(\beta)\le \frac{1-\delta}{\delta}\right) \fstop
			\end{equation}
			In particular, if $\bar \delta^{-1} \notin \N$,  then lower and upper bounds coincide, yielding, for all $\beta\ge  0$,
			\begin{equation}\label{eq:conv-poly}
				d_{\rm TV}(\tau_{\rm start}+ \lfloor \beta \tfrac{n}k\rfloor) - \Pr\left({\rm Poi}(\beta) \le \frac{1-\delta}\delta\right)\overset{\Prob}\longrightarrow0\fstop
			\end{equation}
	\end{proposition}
	
	\begin{remark}\label{rem:eps-bar}
		For simplicity, we stated the convergence result in \eqref{eq:conv-poly} for $\bar \delta^{-1}\notin \N$. Nevertheless, analogous convergence results hold also when $\bar \delta^{-1}\in \N$, provided that there exists some $\varphi$ such that $\log\log(n)\ll\varphi\ll\log(n)$ and
		\begin{equation}\label{HP:technical}
			| j\delta-1|\gg \frac{\varphi}{\log(n)}\,,\qquad \forall j\in\N\,.
		\end{equation}
		Clearly,  under this additional assumption, we have, for any $j\in \N$, the following dichotomy
		\begin{equation}\label{eq:iota2}
			\text{either}\qquad	k^{-j}\ll \frac{e^{-\varphi}}{n}\qquad\text{or}\qquad 	k^{-j}\gg \frac{e^\varphi}{n}\fstop
		\end{equation}
		In particular, under \eqref{HP:technical},  \eqref{eq:conv-poly} holds if  $\delta \uparrow \bar \delta$, whereas $\delta \downarrow \bar \delta$ yields \eqref{eq:conv-poly} with \textquotedblleft $<$\textquotedblright\ in place of \textquotedblleft $\le$\textquotedblright.	
		
		In cases where the condition \eqref{HP:technical} fails to hold, simulations suggest that neither one of the bounds in Proposition \eqref{pr:poly} is sharp; see Figure \ref{fig:eps-1/2} for some simulations when $k=\sqrt n$, to be compared with Figure \ref{fig:eps-07},  which considers  an instance of the case $\bar \delta^{-1}\notin \N$.
	\end{remark}
	\begin{figure}
		\includegraphics[width=6cm]{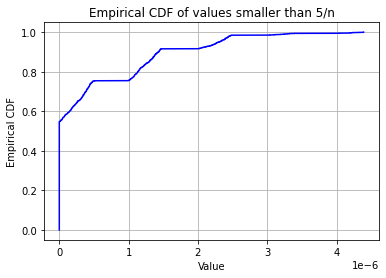}\qquad
		\includegraphics[width=6cm]{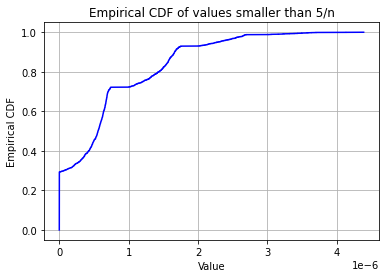}\\
		\includegraphics[width=6cm]{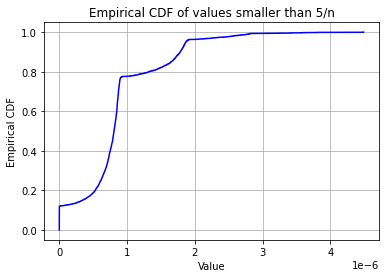}\qquad
		\includegraphics[width=6cm]{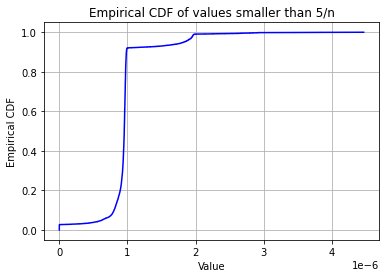}\\
		\includegraphics[width=6cm]{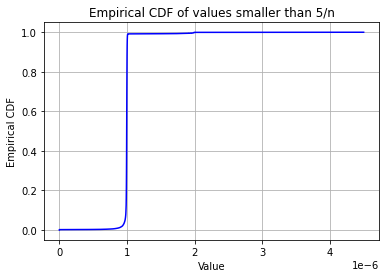}\qquad
		\includegraphics[width=6cm]{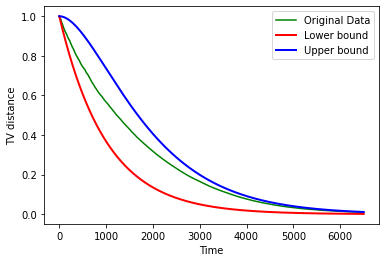}
		\caption{We simulate the process with $n=10^6$ and $k=\sqrt{n}$, starting at $\eta_{0}=\eta^{\rm start}$. We plot the empirical CDF of the entries of $\eta_t$ restricted to the entries for which $\eta_t\le\frac5n$, when $t$ is the first time in which the sequence $d_{\rm TV}(t)$ assumes the values $0.7, 0.5,0.3,0.1,0.01$. The final plot is the behavior of the TV distance and the bounds in Proposition \ref{pr:poly}.}\label{fig:eps-1/2}
	\end{figure}
	
	\begin{figure}
		\includegraphics[width=6cm]{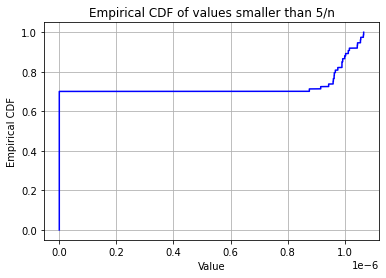}\qquad
		\includegraphics[width=6cm]{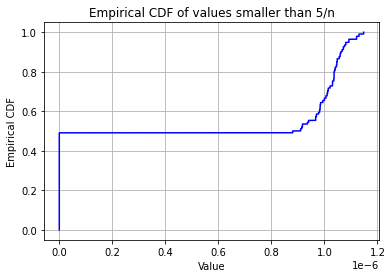}\\
		\includegraphics[width=6cm]{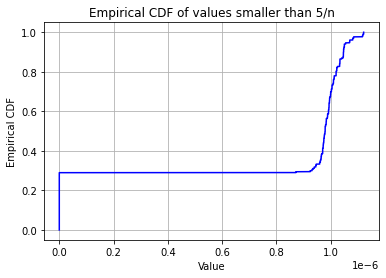}\qquad
		\includegraphics[width=6cm]{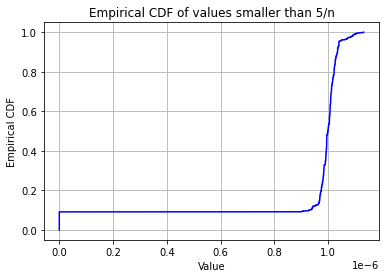}\\
		\includegraphics[width=6cm]{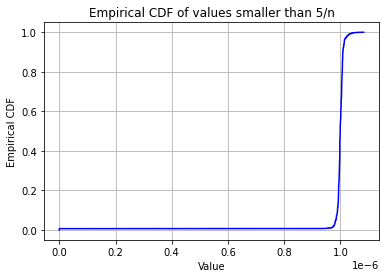}\qquad
		\includegraphics[width=6cm]{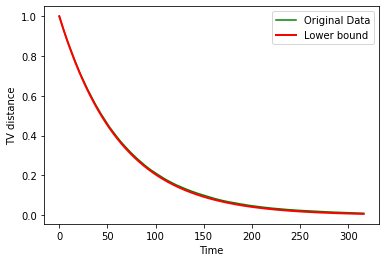}
		\caption{Same as Figure \ref{fig:eps-1/2}, but with $k=\lfloor n^{0.7}\rfloor=15849$. In this case, contrarily to the case $\varepsilon=1/2$ in Figure \ref{fig:eps-1/2}, the empirical mass distribution on the scale $1/n$ is unimodal, having its bulk around $1$.}\label{fig:eps-07}
	\end{figure}

	We split the proof of the above proposition into lower and upper bounds in the two subsequent subsections.

		\subsubsection{Proof of the lower bound in Proposition \ref{pr:poly}}
	For any $t\ge 0$, we employ the notation in Section \ref{suse:general-lb} and call $w_t$ the sub-probability measure on $V$ obtained by removing from $\eta_t$ all the piles of size smaller than $e^{\varphi}/n$, for some prescribed $\log\log(n)\ll\varphi\ll\log(n)$. Thanks to Lemma \ref{lem:glb}, it suffices to prove the following lemma.
		\begin{lemma}\label{lemma:tilde-eta}
			Fix $\beta>0$, $t=\bar t(\beta)\coloneqq\lfloor\beta \tfrac{n}k\rfloor$ and consider the process starting at $\eta_{0}=\eta^{\rm start}$. Then, 
			\begin{equation}
				 |w_t|- \Pr\left({\rm Poi}(\beta)<\frac{1-\delta}{\delta}\right)\overset{\Prob}\longrightarrow  0\fstop
			\end{equation}
		\end{lemma}  
		
		\begin{proof}
			We start by showing that
			\begin{equation}\label{eq:mean-tilde-eta}
				\E[|w_t|]-\Pr\left({\rm Poi}(\beta)<\frac{1-\delta}{\delta}\right)\overset{\Prob}\longrightarrow  0\fstop
			\end{equation}
			By specializing Lemma \ref{lemma:size-pile} to the particular case  $Y\equiv k$, and recalling that $\eta_0=\eta^{\rm start}$, we have
			\begin{align*}
			\E[|w_t|]=\Prob\(|\zeta_t(U)|>\frac{e^\varphi}{n} \)&=\Prob\(k^{-T-1}>\frac{e^\varphi}{n}\)=\Prob\(T< \frac{1-\delta}{\delta}+\frac{\varphi}{\log(k)} \)\,,
			\end{align*}
		where $T\overset{\rm d}{=}{\rm Bin}(t,k/n)$.
			By Le Cam theorem, we have
			\begin{equation}
				\|T-{\rm Poi}(\beta) \|_{\rm TV}\le \sum_{s=1}^t \frac{k^2}{n^2}\le \beta \frac kn\longrightarrow0\comma
			\end{equation}
			hence, being $\varphi\ll\log(k)$, \eqref{eq:mean-tilde-eta} follows. 
			
				Using Chebyshev inequality, we conclude the proof of the lemma as soon as we prove
			\begin{equation}\label{eq:mom2-tilde-eta}
				\E[|w_t|^2]\le \E[|w_t|]^2+o(1)\fstop
			\end{equation}
			The proof of this estimate follows exactly the same lines of arguments employed in the proof of \eqref{eq:aim-2nd-mom}, with a single important difference:  since here we start from $\eta_0=\eta^{\rm start}$, then $\tau_{\rm sep} =0$ (cf.\ \eqref{eq:def-tsep}) with probability $1-\frac1k$; therefore, we can set $t_{\rm sep}=0$ and $\theta =1$ in that proof. (Note that, with these choices, the event $\cE_3$ therein is irrelevant.)		
	\end{proof}
		
		\subsubsection{Proof of the upper bound in Proposition \ref{pr:poly}}
		Fix $\beta>0$, $t=\bar t(\beta)\coloneqq\lfloor\beta \tfrac{n}k\rfloor$, and consider the process starting from $\eta_{0}=\eta^{\rm start}$.
	Define, for $j\ge 1$, $\hat{\eta}^{(j)}$ as the restriction of $\eta_t$ to the piles having size $k^{-j}$.

		For any fixed $C\in\N$, arguing as in the proof of Lemma \ref{lemma:tilde-eta}, and taking a union bound, we get
		\begin{equation}\label{eq:conc-size}
			\max_{j\le C}\bigg||\hat{\eta}^{(j)}|-p_{j-1}(\beta)\bigg|\overset{\Prob}{\to}0\comma\quad \text{with}\  p_{j}(\gamma)\coloneqq\Pr({\rm Poi}(\beta)=j)\fstop
		\end{equation}
		Fix some $\phi$ satisfying  $\log\log(n)\ll\phi\ll\log(n)$, and define
		\begin{equation}\label{eq:iota}
			j_\star:=\max\left\{j\in\N\mid k^{-j}> \frac{e^{\varphi}}{n}  \right\}\fstop
		\end{equation}
		Note that, for this choice of $\varphi$ and since $\bar \delta^{-1}\notin \N$ by assumption, we have
		\begin{equation}
			k^{-j}<\frac{e^{-\varphi}}n\comma\quad j>j_\star\fstop
		\end{equation}
		
	 Now fix any $j\in(j_\star,C]$. We now aim at showing that: for any $\eps>0$,
		\begin{equation}\label{eq:prob-j}
			\lim_{n\to\infty}\Prob\(\exists x\in V\text{ s.t. }{\hat{\eta}^{(j)}(x)}\ge \frac{(1+\eps) p_{j-1}(\beta)}{n} \)=0\fstop
		\end{equation}
		By an application of the union bound and the generalized Markov inequality we deduce that, for any $q\in\N$,
		\begin{equation}\label{eq:gen-mark}
			\Prob\(\exists x\in V\text{ s.t. }{\hat{\eta}^{(j)}(x)}\ge \frac{(1+\eps)p_{j-1}(\beta)}{n} \)\le n\times \frac{\max_{x\in V}\E\bigg[\big({\hat{\eta}^{(j)}(x)}\big)^q\bigg]}{n^{-q}\times  (1+\eps)^q  p_{j-1}(\gamma)^q}\,.
		\end{equation}
		The following lemma, which controls the high moments of  $\hat{\eta}^{(j)}(x)$ for $j> j_\star$, is the key technical ingredient of our proof.
		\begin{lemma}\label{lemma:high-mom}
			Fix $j>j_\star$, $q=\lfloor \log(n)^2\rfloor$, and $\eps>0$. Then, for all $n$ sufficiently large, 
			\begin{equation}\label{eq:q-mom}
				\max_{x\in V}\E\bigg[\big({\hat{\eta}^{(j)}(x)}\big)^q\bigg]\le \left(\frac{(1+\tfrac{\eps}{2})p_{j-1}(\gamma)}n \right)^q\fstop
			\end{equation}
		\end{lemma}
		Observe that \eqref{eq:prob-j} follows at once from \eqref{eq:gen-mark} and \eqref{eq:q-mom}. We postpone the proof of Lemma \ref{lemma:high-mom} to the end of this section, and we now show how the upper bound in Proposition \ref{pr:poly}  follows from \eqref{eq:prob-j}.

				 Fix $\delta>0$. By the triangular inequality we have that, w.h.p., $	d_{\rm TV}(t)$ is bounded above by
		\begin{align}\label{eq:ub-tv}
			& \sum_{j\ge 1}|\hat{\eta}^{(j)} |\left\|\frac{\hat{\eta}^{(j)}}{|\hat\eta^{(j)}|}-\pi \right\|_{\rm TV}\\
			&\le \sum_{j=1}^{j_\star}|\hat{\eta}^{(j)} |+\sum_{j= j_\star+1}^C|\hat{\eta}^{(j)} |\left\|\frac{\hat{\eta}^{(j)}}{|\hat\eta^{(j)}|}-\pi \right\|_{\rm TV}+\sum_{j>C}|\hat{\eta}^{(j)} |\\
			&\le \left(1+\eps\right)\Pr\left({\rm Poi}(\beta)\le \frac{1-\delta}\delta\right)+\sum_{j= j_\star+1}^{C}|\hat{\eta}^{(j)} |\left\|\frac{\hat{\eta}^{(j)}}{|\hat\eta^{(j)}|}-\pi \right\|_{\rm TV}+\Pr({\rm Poi}(\beta)>C)+\delta\comma
		\end{align}
	where in the second line we bounded by $1$ the TV-distances associated to $j\le j_\star$ and $j>C$, whereas in the third one we used \eqref{eq:conc-size} and the following identity  
	\begin{equation}
	\sum_{j=1}^{j_\star}	p_{j-1}(\beta) = \Pr\left({\rm Poi}(\beta)\le j_\star-1\right) = \Pr\left({\rm Poi}(\beta)\le \frac{1-\delta}\delta\right)\comma
	\end{equation}
	which holds true  in view of the definition of $j_\star$ in \eqref{eq:iota} (recall $k=n^\delta$).
	Furthermore,  for any $j> j_\star$,  by \eqref{eq:conc-size} and Lemma \ref{lemma:high-mom},  we obtain, w.h.p.,
		\begin{align}
|\hat{\eta}^{(j)}| \left\|\frac{\hat{\eta}^{(j)}}{|\hat\eta^{(j)}|}-\pi \right\|_{\rm TV}&=\sum_{x\in V}\left[\hat{\eta}^{(j)}(x)- \frac{|\hat{\eta}^{(j)}|}{n}\right]_+\\
&\le\sum_{x\in V}\left[\hat{\eta}^{(j)}(x)- \frac{(1-\eps)p_{j-1}(\beta)}{n}\right]_+\\
&\le\sum_{x\in V}\left[\(1+\frac\eps2\)\frac{p_{j-1}(\beta)}{n}- \frac{(1-\eps)p_{j-1}(\beta)}{n}\right]_+=\frac32\eps p_{j-1}(\beta)\fstop
\end{align}
Plugging this estimate into the right-hand side of \eqref{eq:ub-tv}, we get, for any $\eps\in (0,1)$ and $C\in \N$,  w.h.p.,
\begin{equation}
	\begin{split}
		d_{\rm TV}(t)&
		\le \left(1+\eps\right)\Pr\left({\rm Poi}(\beta)\le \frac{1-\delta}\delta\right)+\frac32\eps+\Pr({\rm Poi}(\beta)>C)+\eps\comma
	\end{split}
\end{equation}
and the desired result follows by taking $C\to\infty$ and $\eps\to 0$.

We are left with the proof of Lemma \ref{lemma:high-mom}.
		\begin{proof}[Proof of Lemma \ref{lemma:high-mom}]
		Recall that $\eta_0=\eta^{\rm start}$, $\beta > 0$, 	$t=\bar t(\beta)\coloneqq\lfloor \beta\frac{n}{k}\rfloor$, $q=\lfloor \log(n)^2\rfloor$, and $j> j_\star$. Consider the collection of random chunks $U_1,\dots,U_q$. For a given $x\in V$, define the events
		\begin{equation}\label{eq:events-J-X}
			\cJ_i=\{|\zeta_t(U_i)|=k^{-j} \}\comma\qquad \cX_i=\{e_t(U_i)=x \}\comma\qquad i\le q\fstop
		\end{equation}
		Note that, quenching  the block dynamics up to time $t$, we write
				\begin{equation}\label{eq:ann-2}
				\big({\hat \eta^{(j)}(x)}\big)^q={\mathsf{P}_{U_1,\dots,U_q}\(\bigcap_{i\le q}\cX_i\cap\cJ_i\)}\comma			\end{equation}
		and taking  expectations on both sides yields
		\begin{align}\label{eq:ann}
			\E&\left[({\hat \eta^{(j)}(x)})^q\right]=\Prob\(\bigcap_{i\le q}\left(\cX_i\cap\cJ_i\right)\)\\
			&=\Prob\(\cJ_1\)\Prob\(\cX_1\mid \cJ_1 \)\prod_{i=1}^{q-1}\Prob\big(\cJ_{i+1} \mid\cap_{\ell\le i}(\cJ_\ell\cap\cX_\ell)\big)\Prob\big(\cX_{i+1}\mid\cJ_{i+1}\cap_{\ell\le i}(\cJ_\ell\cap\cX_\ell)\big)\fstop
		\end{align}
		We start by bounding the first probability on the right-hand side of \eqref{eq:ann}. 
By the same argument employed to show \eqref{eq:mean-tilde-eta}
\begin{equation}
	\Prob(\cJ_1)\sim p_{j-1}(\beta)\,.
\end{equation}
	For the second probability on the right-hand side of \eqref{eq:ann} we claim that, uniformly in $x\in V$, $t\ge 0$ and $j> j_\star$,
\begin{equation}\label{eq:claim}
	\Prob(\cX_1\mid \cJ_1)\le (1+o(1))\, \frac1n \fstop
\end{equation} 
Indeed, letting, for all $h=0,\ldots, j-1$,  $\cY_1(h)$ be the event that the last arrival at $x$ occurs at the $h$-th jump of $U_1$ (here, $\cY_1(0)$ represents the event that $U_1$ starts at $x\in V$ and never leaves this vertex for the subsequent $j$ jumps). Then, 
\begin{align}
	\Prob(\cX_1\mid \cJ_1) &= \sum_{h=0}^{j-1}\Prob(\cX_1\cap \cY_1(h)\mid \cJ_1)\\
	&=\IND_{\eta^{\rm start}(x)>0}\,k^{-j} + \sum_{h=1}^{j-1}
	\Prob(\cX_1\cap \cY_1(h)\mid \cJ_1) \\
	&\le \IND_{\eta^{\rm start}(x)>0}\,k^{-j} + \sum_{h=1}^{j-1} \frac{k-1}{n-1}\frac{1}{k} k^{-(j-1-h)}\comma
\end{align}
where the second identity follows because the probability of starting at $x$ is $\frac1k$, and  never leaving it for $j-1$ jumps is $k^{-(j-1)}$, whereas the second inequality used that the probability of jumping to $x$ starting from somewhere else equals $\frac{k-1}{n-1}\frac1k$. Since we chose $j> j_\star$ and assumed that $\varphi \ll \log(n)$, we have $k^{-j}\ll \frac1n$, ensuring the claim in \eqref{eq:claim}.

Now we aim at proving that: uniformly, for all $i< q$,
\begin{equation}
	\Prob\big(\cJ_{i+1}\mid \cap_{\ell\le i}(\cJ_\ell\cap\cX_{\ell})\big) \le p_{j-1}(\beta)+o(1)\fstop
\end{equation}
The idea is the following: consider the event
\begin{equation}\cQ_{i+1}=\{\nexists (s,\ell)\in\{0,\dots,t\}\times\{1,\dots,i\}\ \text{s.t.}\  e_s(U_{i+1}),e_s(U_{\ell})\in A_s \}\comma
	\end{equation}
and split
\begin{equation}
	\Prob\big(\cJ_{i+1}\mid \cap_{\ell\le i}(\cJ_\ell\cap\cX_{\ell})\big)\le 	\Prob\big(\cQ_{i+1}^c\mid \cap_{\ell\le i}(\cJ_\ell\cap\cX_{\ell})\big)+	\Prob\big(\cJ_{i+1}\cap \cQ_{i+1}\mid \cap_{\ell\le i}(\cJ_\ell\cap\cX_{\ell})\big)\fstop
\end{equation}
In words:
\begin{itemize}
	\item Either the $(i+1)$-th chunk ever jump together with some other chunk among $U_1,\dots, U_i$: this event has probability at most
	\begin{equation}
		\frac{i}{k} + i\left(j-1\right)  \frac{k-1}{n-i}\le \frac{q}{k}+ \left(j-1\right) q\,\frac{k}{n-q}\fstop
	\end{equation}
	Indeed, $\frac{i}{k}$ estimates the probability that $U_{i+1}$ starts at the same location of either one of $U_1,\dots,U_i$; provided that $U_{i+1}$ does not start at any of the initial location of $U_1,\dots,U_i$, the second summand accounts for the probability that there exist a chunk among $U_1,\dots,U_i$ and an update of such a chunk (which are exactly $j-1$ in our conditional probability space) in which also $U_{i+1}$ is involved.
	\item Or, under $\cQ_{i+1}$, $U_{i+1}$ does exactly $j-1$ jumps. Call $\bar p$ the probability of this event, which can be bounded uniformly in $i< q$ as follows
	\begin{equation}
\bar p\le \max_{a\in\{0,\dots, q(j-1) \}}\max_{b\in\{0,\dots, q \}}\Prob\({\rm Bin}(t-a, \tfrac k{n-b})=j-1\)\,.
	\end{equation}
	Indeed, in the conditional probability space and under $\cQ_{i+1}$, at each time $s\le t$ the probability of sampling $e_s(U_{i+1})$ is either zero (if any among $U_1,\dots,U_i$ is involved in the update at time $s$) or bounded above by $\frac{k}{n-(i-1)}$ (if none among $U_1,\dots,U_i$ is involved in the update at time $s$). 
	By the Poisson approximation of Binomials,
	\begin{equation}
	\max_{a\in\{0,\dots, q(j-1)\}}\max_{b\in\{0,\dots, q \}}	\left|\Pr\big({\rm Bin}(t-a, \tfrac k{n-b})=j-1\big)-\Pr\big({\rm Poi}(\beta )=j-1\big) \right|\to 0\,.
	\end{equation}
	Therefore,
	\begin{equation}
		\bar p\le p_{j-1}(\beta)+o(1)\fstop
	\end{equation}
\end{itemize}
		
		We are left with considering the probability
		\begin{equation}
			\Prob\big(\cX_{i+1}\mid\cJ_{i+1}\cap_{\ell\le i}(\cJ_\ell\cap\cX_\ell)\big)\comma\qquad i\in\{1,\dots,q-1\}\fstop
		\end{equation}
		To ease the understanding, we start by considering the case $i=1$.
		For a fixed realization of the trajectory of $U_1$ which satisfies $\cJ_1\cap\cX_1$ and conditionally on $\cJ_2$, the event that $\cX_2$ can occur in three different ways:
		\begin{enumerate}
			\item[(i)] either $U_1$ and $U_2$ are at the same vertex at all times;
			\item[(ii)] or $U_2$ and $U_1$ walk together for a number of jumps $r\in[0,j-1)$, the $r$-th jump put them on different vertices, and then they meet at the $s$-th jump of both $U_1$ and $U_2$, $s\in(r,j-1]$ and they continue their journey together up to reaching $x$ at their $(j-1)$-th jump;
			\item[(iii)] or the two chunks meet and separate more than once.
		\end{enumerate}
		Clearly the conditional probability of the event in (i) can be bounded from above by
		\begin{equation}
			k^{-j}\ll 	 \frac{e^{-\varphi}}n\ll \frac1n\comma
		\end{equation}
		where we simply used \eqref{eq:iota}, $j> j_\star$, and \eqref{eq:iota2}. 
		Alternatively, called $\cG_2$ the event in (ii), we have
		\begin{align}
			\Prob(\cG_2\mid \cJ_1\cap\cX_1\cap\cJ_2)\le\sum_{s=1}^{j-1}  \frac{k-1}{n-1}\times \frac{1}{k}\times k^{-j+1+s}\,,
		\end{align}
		Indeed, given that they are on the same vertex, the probability they remain on the same vertex after a jump is $k^{-1}$. On the other hand, if they are on different vertices, the probability they meet at the next jump of $U_1$ (which is measurable) is $\frac{k-1}{n-1}\times \frac{1}{k}$. Therefore
		\begin{align}
			\Prob(\cG_2\mid \cJ_1\cap\cX_1\cap\cJ_2)&\le  \sum_{s=1}^{j-1} \frac{k-1}{n-1}\times \frac{1}{k}\times k^{-j+1+s}\\
			&\sim\frac{1}{n} \sum_{s=1}^{j-1}  k^{-j+1+s}\sim \frac{1}n\,.
		\end{align}
		Finally, the conditional probability of the event in (iii) can be immediately bounded by
		\begin{equation}\sum_{a=2}^{j/2}\left[\frac{k-1}{n-1}\times\frac1k\right]^a\lesssim n^{-2}\fstop \end{equation}
		In conclusion,
		\begin{equation}
			\Prob(\cX_2\mid \cJ_1\cap\cX_1\cap\cJ_2)\le (1+o(1)) \frac1n\,.
		\end{equation}
		We now argue iteratively. Fixed $i\ge 2$ and a realization of the trajectories of $U_1,\dots,U_{i}$ which satisfies $\bigcap_{\ell\le i}\cJ_\ell\cap\cX_\ell$. Conditionally on these trajectories and $\cJ_{i+1}$, the event $\cX_{i+1}$  can occur in either one of  the following three ways: 
		\begin{enumerate}
			\item[(i)]  for all times $t'\le t$, we have $$e_{t'}(U_{i+1})\in F_{t'}\coloneqq\{y\in V\mid e_{t'}(U_\ell)=y\text{ for some }\ell\le i \}\,;$$
			\item[(ii)]  $U_{i+1}$ walks together with at least one of the other chunks (possibly not always the same) for a number of jumps $r\in[0,j-1)$,  $U_{i+1}$ departs from all the other chunks at the $r$-th jump, and then it meets some $U_{\ell'}$ at the their $s$-th jump, for some $s\in(r,j-1]$, and they continue their journey together up to reaching $x$ at their $(j-1)$-th jump;
			\item[(iii)]   There exists times $0\le t^{(1)}<t^{(2)}<t^{(3)}<t^{(4)}\le t$ such that $e_{t^{(1)}}(U_{i+1})\not\in F_{t^{(1)}}$, $e_{t^{(2)}}(U_{i+1})\in F_{t^{(2)}}$, $e_{t^{(3)}}(U_{i+1})\not\in F_{t^{(3)}}$ and $e_{t^{(4)}}(U_{i+1})\in F_{t^{(4)}}$.
		\end{enumerate}
		Again, the conditional probabilities of the event in (iii) is easy to bound by
		\begin{equation}
		\sum_{a=2}^{j/2}\left[i \times \frac{k-1}{n-i}\times\frac1k\right]^a\ll n^{-1}\,.
		\end{equation}
		Similarly, the probability of the event in (i) can be bounded from above by
		\begin{equation}
			\(\frac{q}{k}\)^j\ll  \frac{q^je^{-\varphi}}{n}\ll \frac1n\,
		\end{equation}
where the asymptotic inequality follows by the choice
		$\varphi\gg \log\log(n)$, $q=\lfloor\log(n)^2\rfloor$, and $j> j_\star$ is constant (cf.\ \eqref{eq:iota} and \eqref{eq:iota2}).
		Concerning the conditional probability of the event in (ii), it can be bounded by
		\begin{align}
		&	\Prob\(\cG_{i+1}\,\bigg\rvert\, \cJ_{i+1}\cap\bigg(\bigcap_{\ell\le i}\cJ_\ell\cap\cX_\ell\bigg)\)\\
		&\qquad\le  \sum_{r=0}^{j-2} (\mathds{1}_{r>0}i+1)k^{-r}\sum_{s=r+1}^{j-1} (\mathds{1}_{s<j-1}i+1)\frac{k-1}{n-i}\times \frac{1}{k}\times k^{-j+1+s}\\
			&\qquad\sim\frac{1}{n-i}\sum_{r=0}^{j-2} k^{-r}(\mathds{1}_{r>0}i+1)\sum_{s=r+1}^{j-1}  k^{-j+1+s}(\mathds{1}_{s<j-1}i+1)\sim \frac1n\comma
		\end{align}
		where we only used $i\le q=\lfloor\log(n)^2\rfloor$. This ends the proof of the lemma.
		\end{proof}

	\section{Sharpness of the conditions in Theorems \ref{th:cutoff} and \ref{th:profile1}}\label{suse:alternative}
	In this section, we discuss a particularly simple  parametric class of block-size distributions of the form
	\begin{equation}\label{eq:par-ex}
		p_X(2)=1-\frac{a}{n}\comma \qquad p_X(  n  )=\frac{a}{n}\comma 
	\end{equation}
	for some positive sequence $a=a_n$, which plays the role of a real parameter. This parameter, running over $[0,n]$, interpolates between the two following extreme cases: (i) when all the blocks have size $2$ (i.e., $a=0$), thus, there is cutoff with (standard) Gaussian profile; (ii) when the blocks all have size $n$ (i.e., $a=n$), so that a single update suffices to reach perfect equilibrium. In particular, case (i) coincides with the model in \cite{chatterjee2020phase}.
	
		Since we are interested in the regime of small $a$,  we assume $a\ll 1$ all throughout this section. Under this condition, we have
	\begin{align}
		\E[X]&\sim 2\\
		\E[X \log X]&\sim 2 \log 2 + a \log n\\
		\E[X(\log X)^2]&\sim 2 (\log 2)^2 + a (\log n)^2 \comma
	\end{align}
	and, thus, 
	\begin{equation}
		\mu\sim \log 2+\frac{a}2\log n\comma \qquad \sigma^2 \sim 
		\frac{a}4(\log n)^2
		\fstop
	\end{equation}
	As a consequence, 
	the entropic product condition in  \eqref{eq:prod-condition2} is satisfied if and only $a\ll 1$, whereas  condition \eqref{eq:HP-noncutoff} is verified if and only if $a\ll (\log n)^{-1}$. Finally, note that condition \eqref{eq:lindeberg} holds true anytime $a\ll (\log n)^{-1}$.	
	
	We analyze  this simple example with two distinct scopes: 
	\begin{itemize}
		\item Section \ref{suse:trichotomy}. We examine the robustness of the cutoff phenomenon. Indeed, our process   corresponds to the regular block-size model in \cite{chatterjee2020phase}, but with the modification that, at any  time and independently of everything else, there is a small probability that the system immediately reaches  equilibrium. We show that a mixing trichotomy (see the dedicated section   below for more details) takes place depending on the scaling of $a$, and comment on the universality of this behavior. 
		\item  Section \ref{suse:linde}. We investigate the sharpness of Theorem \ref{th:profile1} by analyzing the regime $(\log n)^{-2}\lesssim a\ll (\log n)^{-1}$. This example shows that, even if \eqref{eq:prod-condition2} and \eqref{eq:HP-noncutoff} are in force, the convergence in \eqref{eq:statement-profile} might fail in absence of \eqref{eq:CLT} or, equivalently, \eqref{eq:lindeberg} (see Remark \ref{rem:lindeberg}).
	\end{itemize}

		\subsection{Cutoff, half-cutoff, and metastability: a trichotomy}\label{suse:trichotomy} Consider the example in  \eqref{eq:par-ex}. Depending on whether $a$ is much smaller than, larger than, or of the same order as $(\log n)^{-1}$, we  observe three different behaviors:

\smallskip \noindent
 \emph{Cutoff.} As already observed above, if $a\ll \log(n)^{-1}$, then Theorem \ref{th:profile1} applies. In particular, we have cutoff at $\frac{n\log n}{2\log 2}$ with a window of size $n\sqrt{\log n}$ and, since $\rho=\frac\sigma\mu\ll 1$, with a standard Gaussian profile.
	
	\smallskip
	\noindent
	 \emph{Metastability.} If $(\log n)^{-1}\ll a\ll 1$, although the entropic product condition \eqref{eq:prod-condition2} holds, the condition \eqref{eq:HP-noncutoff} fails in this context.	
	Actually,  it is not difficult to check that even the weaker WLLN in \eqref{eq:WLLN} does not hold; therefore, cutoff as in Theorem \ref{th:cutoff} does not occur.
	
		In this case, the random sequence $t\mapsto d_{\rm TV}(t)$ shows a ``cutoff at a random time'', on the scale $t_{\rm ent}\sim\frac{n}{a}\ll n\log n$. Indeed,  w.h.p., the sequence $t\mapsto d_{\rm TV}(t)$ will go from being almost one to be exactly zero in a single random step, coinciding with the  random time
		\begin{equation}
			\tau\coloneqq\{t\ge 1\mid X_t=n \}\fstop
		\end{equation}
		which is geometrically distributed with mean $n/a$. Since this value coincides with the mixing timescale $t_{\rm ent}\sim n/a$, 	
  the sequence $t\mapsto \E[d_{\rm TV}(t)]$ does {\em not} exhibit cutoff. Furthermore, passing to the expected distance-to-equilibrium, we observe an exponential decay: for every $s\ge 0$,   $\E[d_{\rm TV}(s t_{\rm ent})]\sim e^{-s}$. 
		
		This is the clear signature of a metastable behavior:  the system goes to equilibrium because of the realization of a single event occurring at a random exponential time. Furthermore, this example shows that the entropic product condition \eqref{eq:prod-condition2} alone does not suffice to imply cutoff.
	
	\smallskip \noindent
	 \emph{Half-cutoff.} In the intermediate regime $a= c(\log n)^{-1}$, $c> 0$, 
		 the sequence $t\mapsto d_{\rm TV}(t)$  exhibits cutoff at a random point on the scale $n\log n$. However, since $t_{\rm ent}=\frac{n\log n}{2\log 2+c}\asymp n\log n$, such a random time is constrained to be at most $t_{\rm CDSZ}=\frac{n\log n}{2\log 2}$, the cutoff time for the $2$-block model in \cite{chatterjee2020phase}.
		 	  Indeed,  under the event $\tau>(1+\varepsilon)t_{\rm CDSZ}$ (which has probability $e^{-(1+\varepsilon)\frac{c}{2\log 2}}$),  the sequence $d_{\rm TV}(t)$ exhibits cutoff at $t_{\rm CDSZ}$. On the other hand, for all $s\in(0,1)$ and $\delta>0$, 
		\begin{equation}\Prob\big(d_{\rm TV}(s t_{\rm CDSZ})>1-\delta \big)\sim e^{-s\,\frac{c}{2\log 2}}\comma
			\end{equation}
	and therefore, for any sequence $\varepsilon=\varepsilon_n\to 0$ sufficiently slow, we have
		$$\Prob\big(d_{\rm TV}((1-\varepsilon) t_{\rm CDSZ})>1-\delta \,,\, d_{\rm TV}((1+\varepsilon) t_{\rm CDSZ})<\delta\big)\sim e^{-\frac{c}{2\log 2}}\fstop $$
		We conclude that, for any fixed $s\neq 1$,
		\begin{equation}\label{eq:half-cutoff}
			\E[d_{\rm TV}(s t_{\rm CDSZ})]\longrightarrow\IND_{s<1}e^{-s\,\frac{c}{2\log 2}} \fstop
		\end{equation}
\smallskip

	Such a trichotomy cutoff--half-cutoff--metastability is not a peculiar feature of this model, only. Indeed, several other parametric models (see, e.g., \cite{caputo_quattropani_2021_RSA,caputo_quattropani_2021_SPA, avena2019random,avena2022linking}), have been shown to exhibit the very same behavior. 
	In fact, when considering a certain model exhibiting cutoff (such as the case $X= 2$ a.s.) and a parametric perturbation of the model which allows the system to reach equilibrium as consequence of the realization of a certain event which occurs at a random exponential time (such as the event $X=n$, in our case), one should expect the existence of a critical scale for the perturbation parameter in which the two effects (cutoff \textit{vs.}\ exponential decay) coexist, giving rise to  a \emph{half-cutoff} as in \eqref{eq:half-cutoff}.

	\subsection{The role of condition \eqref{eq:lindeberg} in Theorem \ref{th:profile1}}\label{suse:linde}
	Consider again the case $a\ll (\log n)^{-1}$.
	Under this condition, the resulting Block Average process can be coupled,  w.h.p., to the model with $a=0$ up to times $t\gg t_{\rm ent} = \frac{n\log(n)}{2\log(2)}$. As a consequence, both models exhibit cutoff at  $t_{\rm ent}$,  with cutoff  window  and	 standard Gaussian profile as predicted for the model in \cite{chatterjee2020phase}, i.e., \eqref{eq:par-ex} with $a=0$. 
	This complete cutoff phenomenology does not always correspond, however, to the one that Theorems \ref{th:cutoff-window} and \ref{th:profile1} would  indicate. 
	
	Indeed, choosing, e.g.,  $a=\log(n)^{-1-\alpha}$ for some $\alpha\in(0,1)$ in Theorem \ref{th:cutoff-window} would estimate the cutoff window by
	\begin{equation}
	t_{\rm w}\sim 	\left(1+\frac{\sigma}{\mu}\right)\sqrt{\frac{\log n}{\mu}}\asymp (\log n)^{1-\frac\alpha2}\comma
	\end{equation} 
	which is much larger than
	 $\sqrt{\log n}$, the actual size of the cutoff window in \cite{chatterjee2020phase}.	What goes wrong in Theorem \ref{th:profile1} is the validity of \eqref{eq:CLT} or, equivalently, \eqref{eq:lindeberg}. In this sense, conditions \eqref{eq:prod-condition2} and \eqref{eq:HP-noncutoff} without \eqref{eq:CLT} do not suffice to determine the cutoff profile.

	\appendix
	\section{}\label{app:duality}
	For completeness, here we collect some elementary computations from 
	Section \ref{sec:times}.
	
	\subsection{Duality in \eqref{eq:duality}}
For any $x\in V$ and $\eta=\eta_0\in \Delta$, we have
\begin{align}
	&\E[\eta_1(x)]=	\E[\varPsi_A(\eta)(x)]= \sum_{k=2}^n \frac{p_X(k)}{\binom{n}k}\sum_{|B|=k} \varPsi_B  \eta(x)\\
	& = \sum_{k=2}^n \frac{p_X(k)}{\binom{n}k} \sum_{|B|=k} \IND_B(x)\left(\frac1k \sum_{y\in B}\eta(y)\right) + \eta(x)\sum_{k=2}^n \frac{p_X(k)}{\binom{n}k}\sum_{|B|=k}\IND_{B^\complement}(x)\\
	&= \sum_{k=2}^n \frac{p_X(k)}{\binom{n}k}\sum_{|B|=k} \IND_B(x)\left(\frac1k\sum_{y:y\neq x}\IND_B(y)\,\eta(y) + \frac1k\, \eta(x) \right) + \eta(x)\sum_{k=2}^n \frac{p_X(k)}{\binom{n}k}\sum_{|B|=k}\IND_{B^\complement}(x)\\
	&= \sum_{y:y\neq x} \eta(y) \sum_{k=2}^n \frac{p_X(k)}{\binom{n}k k}\sum_{|B|=k}\IND_B(x)\IND_B(y) + \eta(x)\sum_{k=2}^n \frac{p_X(k)}{\binom{n}k}\sum_{|B|=k}\left(\frac1k \IND_B(x) + \IND_{B^\complement}(x)  \right)\\
	&= \sum_{y:y\neq x}\eta(y) \sum_{k=2}^n p_X(x)\frac{\binom{n-2}{k-2}}{\binom{n}k k} + \eta(x)\sum_{k=2}^n p_X(k) \left(\frac{\binom{n-1}{k-1}}{\binom{n}{k} k} + \frac{\binom{n-1}{k}}{\binom{n}{k}}\right)\\
	&= \sum_{y\in V:\,y\neq x}\eta(y) \sum_{k=2}^n p_X(k)\left(\frac{k-1}{n(n-1)}\right) + \eta(x)\sum_{k=2}^n p_X(k)\left(\frac1n+1-\frac{k}n\right)\\
	&= \frac{\E[X]-1}{n(n-1)}\sum_{y\in V:\,y\neq x}\eta(y) + \left(1-\frac{\E[X]-1}{n}\right)\eta(x)\ .
\end{align}
The right-hand side above indeed coincides with 
\begin{equation}\sum_{y\in V}\eta(y)\,P_{\rm RW}(y,x) = \sum_{y\in V}P_{\rm RW}(x,y)\,\eta(y)\fstop
\end{equation}

	\subsection{Proof of Proposition \ref{prop:L2}}\label{app:l2}
		Note that, for every $B\subset V$, we have
		\begin{align}
			\bigg\|\frac{\varPsi_B\eta}{\pi}\bigg\|_2^2-\bigg\|\frac{\eta}{\pi}\bigg\|_2^2 &= \bigg\|\IND_B \frac{\varPsi_B\eta}{\pi}\bigg\|_2^2- \bigg\|\IND_B \frac{\eta}{\pi}\bigg\|_2^2\\
			&= \frac{\eta(B)^2}{\pi(B)}-\bigg\|\IND_B\frac{\eta}{\pi}\bigg\|_2^2\\
			&= n\sum_{x\in B}\eta(x)\,\big(\varPsi_B\eta(x)-\eta(x)\big)\\
			&= \frac{n}{|B|}\sum_{x,y\in B}\eta(x)\, \big(\eta(y)-\eta(x)\big)\ .
		\end{align}
		We further get 
		\begin{align}
		&	\E\bigg[\bigg\|\frac{\eta_1}\pi-1\bigg\|_2^2\bigg]-\bigg\|\frac{\eta}\pi-1\bigg\|_2^2\\
		&= \sum_{k=2}^n p_X(k)\, \frac1{\binom{n}{k}} \sum_{|B|=k} \bigg(\bigg\|\frac{\varPsi_B\eta}{\pi}\bigg\|_2^2-\bigg\|\frac{\eta}{\pi}\bigg\|_2^2\bigg)\\
			&= \sum_{k=2}^n p_X(k)\, \frac1{\binom{n}k} \sum_{|B|=k} \frac{n}{k}\sum_{x,y\in B}\eta(x)\, \big(\eta(y)-\eta(x)\big)\\
			&= \sum_{x,y \in V} \sum_{k=2}^n p_X(k)\, \frac1{\binom{n-1}{k-1}} \sum_{|B|=k}\IND_B(x)\IND_B(y)\, \eta(x)\, \big(\eta(y)-\eta(x)\big)\\
			&= \sum_{\substack{x,y\in V\\
					x\neq y}}\eta(x)\, \big(\eta(y)-\eta(x)\big)\sum_{k=2}^n p_X(k)\, \frac1{\binom{n-1}{k-1}}\sum_{|B|=k}\IND_B(x)\IND_B(y)\\
			&= \sum_{x,y\in V}\eta(x)\, \big(\eta(y)-\eta(x)\big) \sum_{k=2}^n p_X(k)\, \frac{\binom{n-2}{k-2}}{\binom{n-1}{k-1}}\\
			&= \frac{\E[X]-1}{n-1}\sum_{x,y\in V}\eta(x)\, \big(\eta(y)-\eta(x)\big)\\
			&= -\frac1{t_{\rm rel}}\, \bigg\|\frac{\eta}\pi-1\bigg\|_2^2\comma
		\end{align}
		where for the last identity we used the definition in \eqref{eq:t-rel}.	
		Hence, we obtain
		\begin{equation}
		\E\left[\bigg\|\frac{\eta_1}\pi-1\bigg\|_2^2\right] = \bigg(1-\frac1{t_{\rm rel}}\bigg)\bigg\|\frac{\eta}\pi-1\bigg\|_2^2 \ ,
		\end{equation}
		which, by iterating and using the Markov property,  concludes the proof of Proposition \ref{prop:L2}.

		\subsection{Proof of Proposition \ref{pr:movas}}\label{app:movas}
	Observe that, for each block $B\subset V$, we have
		\begin{equation}
			D(\varPsi_B\eta\|\pi)-D(\eta\|\pi)= -\eta(B)\,  {\rm ent}_B(\eta)\ge -\eta(B) \log |B|\comma\quad \text{with}\ \eta(B)=\sum_{y\in B}\eta(y)\fstop
		\end{equation}
		Here, $|B|$ denotes the cardinality of $B\subset V$, while ${\rm ent}_B(\eta)\in [0,\log |B|]$  the relative entropy of $\eta\,\IND_B $
		with respect to the uniform measure on $B$. Hence, we have
		\begin{align}
\E\left[D(\eta_1\|\pi)\right]-D(\eta\|\pi)&= \sum_{k=2}^n p_X(k)\,\frac1{\binom{n}{k}}\sum_{|B|=k}\left(D(\varPsi_B\eta\|\pi)-D(\eta\|\pi)\right)\\
&\ge -\sum_{k=2}^n p_X(k)\, \frac{\log k}{\binom{n}{k}}\sum_{|B|=k}\eta(B)\\
&= -\sum_{k=2}^n p_X(k)\, \log k\, \frac{\binom{n-1}{k-1}}{\binom{n}k} = -\frac1n\, \E[X\log X]\fstop
		\end{align}
		Iterating this inequality, we obtain the desired claim.
		
		\begin{acknowledgement}
			M.Q.\ and F.S.\  are members of GNAMPA-INdAM. M.Q.\ thanks the
			German Research Foundation (project number 444084038, priority program SPP2265)
			for financial support.
			F.S.\   acknowledges financial support by “Microgrants 2022”, funded by Regione FVG, legge LR 2/2011.
		\end{acknowledgement}

		\bibliographystyle{alpha}
	\newcommand{\etalchar}[1]{$^{#1}$}

	\end{document}